\documentclass[11pt]{article}
\usepackage{amssymb, latexsym, mathrsfs, amscd}
\newtheorem{defin}{}
\newtheorem{saetze}[defin]{}
\newtheorem{conjec}[defin]{}
\newtheorem{lemmas}[defin]{}
\newtheorem{remarc}[defin]{}
\newtheorem{folger}[defin]{}

\newenvironment{theorem}  {\begin{saetze}\it {\bf Theorem:}}{\end{saetze}}
\newenvironment{conjecture}{\begin{conjec}\it {\bf Conjecture:}}{\end{conjec}}
\newenvironment{lemma}    {\begin{lemmas}\it {\bf Lemma:}}{\end{lemmas}}
\newenvironment{corollary}{\begin{folger}\it {\bf Corollary:}}{\end{folger}}
\newenvironment{remark}   {\begin{remarc}\it {\bf Remark:}}{\end{remarc}}
\newenvironment{proof}    {{\it Proof}:}{{\hfill \fillbox \bigskip}}

\newenvironment{items}{\begin{list}{$\alph{item})$}
{\labelwidth18pt \leftmargin18pt \topsep3pt \itemsep1pt \parsep0pt}}
{\end{list}}

\newcommand{\fillbox}{\mbox{$\bullet$}}

\newcommand{\A}{\mathcal A}    % Quillen cat.
\renewcommand{\L}{\mathcal{L}} % elem. ab. subgroups
\newcommand{\C}{\mathcal{C}}   % complements to $L$
\renewcommand{\O}{\mathcal{O}} % central subgroups
\newcommand{\cT}{\mathcal{T}}  % objects

\newcommand{\N}{\mathbb{N}}
\newcommand{\Z}{\mathbb{Z}}

\newcommand{\ra}{\rightarrow}
\newcommand{\Ra}{\Rightarrow}
\newcommand{\ms}{\mapsto}

\newcommand{\ol}{\overline}
\newcommand{\ul}{\underline}

\newcommand{\Ext}{\mathrm{Ext}}
\newcommand{\Zeta}{\hat{\zeta}}

\begin{document}

\title {The Quillen category of \\ finite $p$-groups and coclass theory}
\author{Bettina Eick and David J. Green}
\date{2 October 2013}
\maketitle

\begin{abstract}
Coclass theory can be used to define infinite families of finite $p$-groups 
of a fixed coclass. It is conjectured that the groups in one of these 
infinite families all have isomorphic mod-$p$ cohomology rings. Here we 
prove that almost all groups in one of these infinite families have equivalent 
Quillen categories. We also show how the Quillen categories of the groups
in an infinite family are connected to the Quillen category of their
associated infinite pro-$p$-group of finite coclass. 
\end{abstract}

\section{Introduction}

The coclass of a finite $p$-group of order $p^n$ and nilpotency class
$c$ is defined as $n-c$. Leedham-Green and Newman \cite{LGN} proposed the
use of coclass a primary invariant for the classification and 
investigation of finite $p$-groups. This suggestion was highly successful
and has led to various new insights into the structure of finite $p$-groups.
We refer to the book by Leedham-Green and McKay \cite{LGM} for details.

Eick and Leedham-Green \cite{ELG08} introduced the {\em coclass families}:
these are certain infinite families of finite $p$-groups of fixed coclass;
see Section \ref{coclass} for the explicit definition. The groups in a 
coclass family have a similar structure and hence can be treated 
simultaneously in various structural investigations of finite $p$-groups. 
The significance of the coclass families is underlined by the result that 
the infinitely many finite $2$-groups of fixed coclass fall into finitely 
many coclass families and finitely many other groups.

The Quillen category $\A_p(G)$ of a finite group $G$ is the category
whose objects are the elementary abelian $p$-subgroups of $G$ and whose
morphisms are the injective group homomorphisms induced by conjugation 
with elements of $G$. The following is the main result of this paper, 
see Section \ref{equiv} for a proof.

\begin{theorem}
\label{thm:main}
Let $(G_x \mid x \geq 0)$ be a coclass family of finite $p$-groups.
Then there exists $x_0 \in \N$ so that the Quillen categories $\A_p(G_x)$ 
and $\A_p(G_y)$ are equivalent for all $x,y \geq x_0$.
\end{theorem}

For the proof of Theorem \ref{thm:main} we consider $x \in \N$ large enough
and define a special skeleton $\ol{\A}_p(G_x)$ for the Quillen category of
$G_x$ together with an explicit functor $F : \ol{\A}_p(G_x) \ra
\ol{\A}_p(G_{x+1})$. We then show that this induces an isomorphism of
categories. Hence the categories $\A_p(G_x)$ and $\A_p(G_{x+1})$ are
equivalent.

As a by-product we also define a special skeleton $\ol{\A}_p(S)$ for the
infinite pro-$p$-group $S$ associated to the coclass family $(G_x \mid x
\geq 0)$. For each large enough $x$ we then introduce a functor $F_S :
\ol{\A}_p(G_x) \ra \ol{\A}_p(S)$. This functor $F_S$ is not necessarily
injective or surjective on objects, but nonetheless it exhibits an
interesting link between the Quillen category of an infinite pro-$p$-group
of finite coclass and its associated coclass families.

Given the Quillen category $\A_p(G)$ of a finite $p$-group $G$ and given 
a field $k$ of characteristic $p$, we denote 
\[ \ol{H}^*(G,k) = \lim_{E \in \A_p(G)} H^*(E,k). \]
Restriction induces a natural homomorphism
\[ \phi_G : H^*(G,k) \rightarrow \ol{H}^*(G,k). \]
Quillen \cite[Th. 6.2]{Quillen} proved
that~$\phi_G$ is an \emph{inseparable isogeny}\@. That is,
% that the kernel and image of $\phi_G$ have some interesting properties:
every homogeneous element of the kernel of 
$\phi_G$ is nilpotent and for every element $s$ in the range of $\phi_G$ 
there exists $n \in \N$ so that $s^{p^n}$ is an element of the image of 
$\phi_G$. Thus $\phi_G$ induces a homeomorphism of the prime ideal spectra of
$H^*(G, k)$ and $\ol{H}^*(G,k)$ and it shows how the Quillen category of 
a finite $p$-group influences its mod-$p$ cohomology ring. Theorem 
\ref{thm:main} has the following immediate corollary.

\begin{corollary}
\label{cor:main}
Let $(G_i \mid i \geq 0)$ be a coclass family of finite $p$-groups.
Then there exists $l \in \N$ so that $\ol{H}^*(G_i,k) \cong \ol{H}^*(G_j,k)$ 
for all $i,j \geq l$.
\end{corollary}

Carlson \cite{Car05} proved that the mod-$2$ cohomology rings of the 
finite $2$-groups of a fixed coclass fall into finitely many isomorphism 
types. The following conjecture expands this result to odd primes $p$
and it also strengthens it for prime $2$. Corollary \ref{cor:main} 
can be considered as an indication that supports the following conjecture.

\begin{conjecture}
\label{con:main}
Let $(G_i \mid i \geq 0)$ be a coclass family of finite $p$-groups.
Then there exists $l \in \N$ so that $H^*(G_i,k) \cong H^*(G_j,k)$
holds for all $i,j \geq l$.
\end{conjecture}

%%%%%%%%%%%%%%%%%%%%%%%%%%%%%%%%%%%%%%%%%%%%%%%%%%%%%%%%%%%%%%%%%%%%%%%%%%%%%

\section{Cohomology, extensions and complements}
\label{cohext}

In this section we recall some of the well-known results about
cohomology groups and their applications in group theory. We refer
to \cite{Benson:I} and \cite{Benson:II} for background and proofs.

Let $G$ be a group acting on the module $M$. A normalised $n$-cochain of
$G$ to $M$ is a map $\gamma : G \times \ldots \times G \ra M$ satisfying
that $\gamma(g_1, \ldots, g_n) = 0$ if $g_i = 1$ for some $i \in \{1, \ldots,
n\}$. Let $C^n(G,M)$ denote the set of all normalised $n$-cochains and
define

\[ \Delta_n : C^n(G,M) \ra C^{n+1}(G,M) : \gamma \ms \Delta_n(\gamma) 
   \;\;\; \mbox{ with } \]
\begin{eqnarray*}
&& \Delta_n(\gamma)(g_1, \ldots, g_{n+1}) := \\
&& \hspace{1cm} \gamma(g_2, \ldots, g_{n+1}) +
   \sum_{i=1}^n (-1)^i \gamma(g_1, \ldots, g_i g_{i+1}, \ldots, g_n) 
   + (-1)^{n+1} \gamma(g_1, \ldots, g_n)^{g_{n+1}}. 
\end{eqnarray*}

Let $Z^n(G,M)$ denote the kernel of $\Delta_n$ and let $B^n(G,M)$ be the
image of $\Delta_{n-1}$. An element of $Z^n(G,M)$ is called an $n$-th
cocycle and an element of $B^n(G,M)$ is an $n$-th coboundary. Further,
the set $Z^n(G,M)$ is an abelian group with subgroup $B^n(G,M)$. Their
quotient $H^n(G,M)$ is the $n$-th cohomology group. The coset $\gamma + 
B^n(G,M)$ for $\gamma \in Z^n(G,M)$ is also denoted by $[\gamma]$.

\subsection{Extensions}

We recall the well-known connection between the 2-dimensional cohomology
of groups and extensions. Let $G$ be a finite group acting on a module $M$. 
An extension of $G$ by $M$ is a group $E$ having a normal subgroup $E_M
\cong M$ with $E/E_M \cong G$. Equivalently, the group $E$ satisfies a 
short exact sequence $0 \ra M \ra E \ra G \ra 1$.

For $\tau \in Z^2(G,M)$ let $\Ext(\tau)$ be the set $G \times M$ equipped
with the multiplication 
\[ (g,m)(h,n) = (gh, m^h + n + \tau(g,h)) \mbox{ for } g,h \in G 
\mbox{ and } m, n \in M. \]
Then $\Ext(\tau)$ is a group and their exists the natural epimorphism 
$\epsilon(\tau) : \Ext(\tau) \ra G : (g,m) \ms g$ with kernel $\{(1,m) 
\mid m \in M \} \cong M$. Thus $\Ext(\tau)$ is an extension of $G$ by $M$.
It is well-known that every extension of $G$ by $M$ is isomorphic to a 
group $\Ext(\tau)$ for some $\tau \in Z^2(G,M)$. 

Let $\tau, \sigma \in Z^2(G,M)$ with $\tau - \sigma = \gamma \in B^2(G,M)$.
Let the coboundary $\gamma$ be induced by the map $\zeta : G \ra M$ with
$\zeta(1) = 0$; that is, let $\gamma(g,h) = \zeta(g)^h + \zeta(h) - \zeta(gh)$
for $g, h \in G$. Then $\Ext(\tau)$ is isomorphic to $\Ext(\sigma)$ via
\[ \Ext(\tau) \ra \Ext(\sigma) : (g,m) \ms (g, m+\zeta(g)).\]

Hence we obtain that up to group isomorphism every extension of $G$ by $M$ 
is defined by some element $[\gamma] \in H^2(G,M)$. Note that different 
elements of $H^2(G,M)$ can define isomorphic extension groups so that the
isomorphism problem for extension groups is not completely solved by this
construction.

\subsection{Complements}
\label{compl}

We recall the well-known connection between the 1-dimensional cohomology
of groups and complements.
Let $G$ be a finite group acting on a module $M$, let $\tau \in Z^2(G,M)$
and denote $E = \Ext(\tau)$ and $\epsilon = \epsilon(\tau)$. Further, we
identify $M$ with $\{(1,m) \mid m \in M\}$. Then there exists a complement 
to $M$ in $E$ if and only if $\tau \in B^2(G,M)$. If a complement $C$ 
exists, then it is isomorphic to $G$ via $\epsilon$. Let 
$c(g) \in C$ with $\epsilon(c(g)) = g$. Then
\[ C = \{ c(g) \mid g \in G\} \mbox{ and } c(g) = (g,m_g) \mbox{ for some }
m_g \in M. \]

As $C$ is a subgroup of $E(\tau)$, it follows that $c(gh) = c(g) c(h)$ holds,
or, equivalently, $\tau(g,h) = m_{gh} - m_g^h - m_h$. Thus given $C$ and an
element $\delta \in Z^1(G,M)$, it follows that
\[ C(\delta) = \{ (g, m_g + \delta(g)) \mid g \in G \} \]
is another complement to $M$ in $E$. It is well-know that every complement 
to $M$ in $E$ is of the form $C(\delta)$ for some $\delta \in Z^1(G,M)$ and
hence there is a one-to-one correspondence between the elements of $Z^1(G,M)$ 
and the set of complements to $M$ in $E$. 

Let $\delta, \rho \in Z^1(G,M)$. Then $C(\delta)$ and $C(\rho)$ are 
conjugate via $m \in M$ if and only if $\delta - \rho = \gamma \in B^1(G,M)$
and $\gamma$ is induced by $m$ via $\gamma(g) = m^g - m$. In this case
we thus obtain that 
\[ (g, m_g + \delta(g))^{(1,m)} = (g, m_g + \rho(g)). \]

\subsection{An action}

For $g \in G$ and $\gamma \in Z^n(G, M)$ we denote
\[ \gamma^g (g_1, \ldots, g_n) = \gamma(g_1^g, \ldots, g_n^g)^{g^{-1}}.\]
It is technical, but straightforward to observe that $\gamma^g \in Z^n(G, M)$ 
for each $g \in G$ and each $\gamma \in Z^n(G, M)$. Further, this notation 
defines a {\em left} action of $G$ on $Z^n(G, M)$. Moreover, if $\gamma \in 
B^n(G, M)$, then $\gamma^g \in B^n(G, M)$ follows. More precisely, the group
$G$ acts on $Z^n(G, M)$ via $\gamma \ms \gamma^{g^{-1}}$. This action leaves 
$B^n(G, M)$ setwise invariant and hence induces an action of $G$ on 
$H^n(G, M)$.

\subsection{Mappings}

Throughout this paper we will consider various maps on groups of cocycles 
and in this section we want to establish some notation.

If $L \leq G$, then the restriction from $G$ to $L$ induces the homomorphism
\[ res_L : Z^n(G, M) \ra Z^n(L,M) : \gamma \ms \gamma_L.\]
This maps $B^n(G,M)$ onto $B^n(L,M)$ and hence induces a homomorphism 
$H^n(G,M) \ra H^n(L, M)$. Note that restriction is compatible with the
action of $G$ on cocycles.

If $G$ acts on two modules $N$ and $M$ and $hom : M \ra N$ is a $G$-module
homomorphism, then we denote with the same name $hom : Z^n(G, M) \ra 
Z^n(G, N)$ the induced homomorphism on cocycles. The induced $hom$ is 
compatible with the action of $G$ on cocycles and maps $B^n(G, M)$ into 
$B^n(G, N)$. Hence it induces a homomorphism $H^n(G, M) \ra H^n(G, N)$
which we also denote with $hom$. We use this general concept in 
various cases.

First, if $N \leq M$, then the projection $pro_{M/N} : M \ra M/N$ is a 
$G$-module homomorphism. Hence we obtain
\[ pro_{M/N} : Z^n(G, M) \ra Z^n(G, M/N).\]

Projection is not necessarily injective or surjective, but it maps 
the group of coboundaries $B^n(G,M)$ onto $B^n(G,M/N)$.

Next, if $N \leq M$, then the inclusion $inc_N : N \ra M$ is a
$G$-module homomorphism. It yields
\[ inc_N : Z^n(G, N) \ra Z^n(G, M). \]

Inclusion is injective, but not necessarily surjective. It maps $B^n(G,N)$
into $B^n(G, M)$, but this restriction to coboundaries is also not 
necessarily surjective.

Finally, consider $T \cong \Z_p^d$ where $\Z_p$ denotes the $p$-adic integers
and suppose that this a $G$-module. Then division by $p^l$ yields an 
$G$-module isomorphism $div_l : p^l T \ra T$ which induces an isomorphism 
\[ div_l : Z^n(G, p^l T) \ra Z^n(G, T).\]
Similarly, for every $k \geq l$ the mapping $div_l$ yields an isomorphism 
$div_l : p^l T / p^k T \ra T / p^{k-l} T$ and this induces an isomorphism 
\[ div_l : Z^n(G, p^lT/p^kT) \ra Z^n(G, T/p^{k-l}T).\]

As $div_l$ is an isomorphism, it maps the group of coboundaries to the
group of coboundaries of the image and thus also induces an isomorphism
of the corresponding cohomology groups.

\section{Infinite pro-$p$-groups of finite coclass}
\label{infinite}

Let $S$ be an infinite pro-$p$-group of finite coclass $r$. In this section
we investigate the Quillen category $\A_p(S)$. As a preliminary step, we
first recall the well-known structure of $S$. We refer to \cite{LGM} for
background and proofs. 

As shown in \cite[7.4.13]{LGM}, the group $S$ is a so-called uniserial
$p$-adic pre-space group of some dimension $d$. This asserts that there
exists a normal subgroup $T$ in $S$ with $T \cong \Z_p^d$, where $\Z_p$
denotes the $p$-adic integers, and $S/T$ is a finite $p$-group of coclass 
$r$. Further, the series defined by $T_0 := T$ and $T_{i+1} = [T_i, S]$ 
satisfies $[T_i : T_{i+1}] = p$ for $i \in \N_0$. This implies that if
$N$ is an arbitrary $S$-normal subgroup of $T$, then $N = T_i$ for some
$i \in N_0$. In particular, $p T_i = T_{i+d}$ for every $i \in \N_0$. 

A subgroup $T$ with these properties is called a {\em translation subgroup} 
of the infinite pro-$p$-group $S$. Note that  the translation subgroup of 
$S$ is not unique; In fact, if $T$ is a translation subgroup of $S$, then 
every of its subgroups $T_i$ is also a translation subgroup of $S$. Hence 
there are infinitely many translation subgroups for $S$. If $T$ is a 
translation subgroup of $S$ and $S/T$ has nilpotency class $\ell-1$, then 
the $\ell$-th lower central series subgroup $\gamma_\ell(S)$ satisfies 
$\gamma_\ell(S) \leq T$ and thus every subgroup $\gamma_{\ell+k}(S)$ for 
$k \in \N_0$ is a possible translation subgroup of $S$. We thus can choose 
the translation subgroup of $S$ as large enough subgroup of the lower 
central series of $S$.

We now fix $T$ as $\gamma_\ell(S)$ for some large enough $\ell$ so that
$T$ is a translation subgroup of $S$. Define $P = S/T$ and choose $\rho
\in Z^2(P,T)$ with $\Ext(\rho) \cong S$. In the following, we identify $S$ 
with $\Ext(\rho)$ and we denote with $\epsilon = \epsilon(\rho)$ the 
natural projection $S \ra S/T = P$ associated with the extension structure
of $S$.

We now use this setup to investigate the Quillen category $\A_p(S)$.

\subsection{The objects of $\A_p(S)$}

The objects of $\A_p(S)$ are the elementary abelian subgroups of $S$. In 
this section we aim for an alternative, more explicit description of these 
objects.

For $L \leq P$ let $\ol{L} \leq S$ be the full preimage under of $L$ under
$\epsilon$. We denote
\begin{eqnarray*}
\L 
&=& \{ L \leq P \mid L \mbox{ elementary abelian with } \rho_L \in B^2(L,T)\} \\
&=& \{ L \leq P \mid L \mbox{ elementary abelian with } [\rho_L] = 0 \} \\
&=& \{ L \leq P \mid L \mbox{ elementary abelian, and }  
       \ol{L} \mbox{ splits over } T \}
\end{eqnarray*}
For each $L \in \L$ we choose a fixed complement $C_L$ to $T$ in $\ol{L}$.
This has the form $C_L = \{ c_L(l) \mid l \in L \}$ with $c_L(l) = (l, 
t_L(l))$ for some $t_L(l) \in T$. Note that $c_L(lh) = c_L(l) c_L(h)$, or, 
equivalently, $\rho(l,h) = t_L(lh) - t_L(l)^h - t_L(h)$ holds. In other words,
$\rho = -\Delta_1(t_L)$. 

For $\delta \in Z^1(L,T)$ we then denote $C_L(\delta) = \{ (l, t_L(l) + 
\delta(l)) \mid l \in L \}$ and we recall that this is also a complement
to $T$ in $\ol{L}$. Further, let
\[\C(L) = \{ C_L(\delta) \mid \delta \in Z^1(L,T) \}. \]
Then $\C(L)$ is the set of all complements to $T$ in $\ol{L}$. We now
obtain the following description of the objects in $\A_p(S)$.

\begin{theorem}
The objects in the Quillen category $\A_p(S)$ are the disjoint union
of the sets $\C(L)$ for $L \in \L$.
\end{theorem}

\begin{proof}
Let $E$ be an object in $\A_p(S)$. Then $E$ is an elementary abelian
subgroup of $S$. Thus $E \cap T = \{1\}$, as $T$ is torsion free. Let
$L = ET/T$. Then $L \in \L$, as $L \cong E$ and $E$ is a complement
to $T$ in $\ol{L}$. Thus $E = C_L( \delta )$ for some $\delta \in 
Z^1(L,T)$, as every complement to $T$ in $\ol{L}$ has this form. 
The converse is obvious.
\end{proof}

\subsection{The morphisms of $\A_p(S)$}

Each morphism of $\A_p(S)$ is induced by conjugation with an element of $S$. 
More precisely, given two objects $C_L(\delta)$ and $C_H(\sigma)$ in 
$\A_p(S)$ and an element $g \in S$, then $g$ induces a morphism $\mu_g : 
C_L( \gamma) \ra C_H( \sigma)$ if and only if $C_L( \gamma)^g \leq 
C_H( \sigma)$ holds. Our next aim is to given an alternative characterisation
for this setting.

For $L, H \in \L$ and $g \in S$ with $L^g \leq H$ we define
\[ \zeta_{L,H,g} : L \ra S : l \ms c_H(l^g)^{-g^{-1}} \cdot c_L(l) \]
and for $\sigma \in Z^1(H, T)$ we denote
\[ \sigma^g_L : L \ra T : l \ms \sigma(l^g)^{g^{-1}} \in Z^1(L,T).\]

\begin{theorem} \label{morph}
Let $C_L( \gamma)$ and $C_H( \sigma)$ are two objects in $\A_p(S)$ and
let $g \in S$. Then $C_L( \gamma)^g \leq C_H( \sigma)$ if and only if 
$L^g \leq H$ and $\sigma^g_L -  \gamma = \zeta_{L,H,g}$.
\end{theorem}

\begin{proof}
We use the identification $S = \Ext(\rho)$ with the natural epimorphism 
$\epsilon : S \ra P : (w, t) \ms w$.

First suppose that $C_L( \gamma)^g \leq C_H( \sigma)$. Then this conjugation
is compatible with $\epsilon$ and thus $L^g \leq H$ follows directly. Further,
the elements of $C_L(\gamma)$ are 
$(l, t_L(l) + \gamma(l))$ for $l \in L$
and the elements of $C_H(\sigma)$ are 
$(h, t_H(h) + \sigma(h))$ for $h \in H$. 
The element $(l, t_L(l) + \gamma(l))^g$ maps onto $l^g$ under $\epsilon$ 
and is an element in $C_H(\sigma)$ by assumption. Hence
$(l, t_L(l) + \gamma(l))^g = (l^g, t_H(l^g) + \sigma(l^g))$. 
Recall that $c_L(l) = (l, t_L(l))$ and $c_H(h) = (h, t_H(h))$. Thus 
$c_L(l)^g (1, \gamma(l)^g) 
  = (c_L(l) (1, \gamma(l))^g 
  = (l, t_L(l)+\gamma(l))^g 
  = (l^g, t_H(l^g) + \sigma(l^g))
  = c_H(l^g) (1, \sigma(l^g))$. 
This translates to
$c_H(l^g)^{-1} c_L(l)^g = (1, \sigma(l^g) - \gamma(l)^g)$.
Conjugating the right and left hand side of this equation with $g^{-1}$ 
yields $\zeta_{L,H,g}(l) = \sigma(l^g)^{g^{-1}} - \gamma(l)$. 

The converse follows with similar arguments.
\end{proof}

We investigate $\zeta_{L,H,g}$ in more detail as preparation for our
later sections. For $L, H \in \L$ and $w \in P$ let
\[ \Zeta_{L,H,w} : L \ra T : l \ms 
      (t_L(l)^w -t_H(l^w) + \rho(l,w) - \rho(w, l^w))^{w^{-1}} \]
and for $t \in T$ let $\lambda_t$ be the coboundary in $B^1(L, T)$ 
induced by $t$ via $\lambda_t : L \ra T : l \ms [t, l]$. Then 
$\lambda_t^w = \lambda_{t^{w^{-1}}} \in B^1(L,T)$ as
\[ \lambda_t^w : L \ra T : l \ms [t, l^w]^{w^{-1}} = [t^{w^{-1}}, l]. \] 

\begin{lemma} \label{zeta}
Let $L, H \in \L$ and $g = (w,t) \in S$ with $L^g \leq H$. 
\begin{items}
\item[\rm (a)]
$\zeta_{L,H,g} = \Zeta_{L,H, w} - \lambda_t^w$.
\item[\rm (b)]
if $\zeta_{L, H, g} \in Z^1(L, T)$, then $\zeta_{L, H, h} \in Z^1(L, T)$
for all $h \in gT$.
\item[\rm (c)]
there exist $\gamma$ and $\sigma$ with $C_L(\gamma)^g \leq C_H(\sigma)$
if and only if $\zeta_{L,H,g} \in Z^1(L,T)$.
\item[\rm (d)]
$g$ centralises $C_L(\gamma)$ if and only if $w$ centralises $L$ and
$\zeta_{L,L,g} = \gamma^g - \gamma$.
\end{items}
\end{lemma}

\begin{proof}
(a) 
Let $l \in L$ and denote $k = l^w$. Note that $g$ acts on $P$ and on $T$
as $w$. Then
\begin{eqnarray*}
   && \zeta_{L,H,g}(l)^g  \\
   &=& c_H(k)^{-1} \cdot c_L(l)^g \\
   &=& (k, t_H(k))^{-1} \cdot (l, t_L(l))^g \\
   &=& (k^{-1}, -t_H(k)^{k^{-1}} - \rho(k,k^{-1}))
        \cdot g^{-1} (l, t_x(l)) g \\
   &=& (k^{-1}, -t_H(k)^{k^{-1}} - \rho(k,k^{-1}))
        \cdot (w,t)^{-1} (l, t_L(l)) (w,t) \\
   &=& (k^{-1}, -t_H(k)^{k^{-1}} - \rho(k,k^{-1})) 
     \cdot (w^{-1}, -t^{w^{-1}} - \rho(w,w^{-1})) (l, t_L(l)) (w,t) \\
   &=& (k^{-1}, -t_H(k)^{k^{-1}} - \rho(k,k^{-1})) 
      \cdot (w^{-1}, -t^{w^{-1}} - \rho(w,w^{-1}))
        (lw, t_L(l)^w + t + \rho(l,w)) \\
   &=& (k^{-1}, -t_H(k)^{k^{-1}} - \rho(k,k^{-1})) \\
   &&  \;\;\;\;\; \cdot
        (w^{-1}lw, (-t^{w^{-1}} - \rho(w,w^{-1}))^{lw}
              + t_L(l)^w + t + \rho(l,w) + \rho(w^{-1}, lw)) \\
   &=& (k^{-1} w^{-1}lw, (-t_H(k)^{k^{-1}} - \rho(k,k^{-1}))^{w^{-1}lw}
         + (-t^{w^{-1}} - \rho(w,w^{-1}))^{lw}   \\
   && \;\;\;\;\; + t_L(l)^w + t + \rho(l,w) + \rho(w^{-1}, lw)
         + \rho(k^{-1}, w^{-1}lw)) \\
   &=& (1, -t_H(k) - \rho(k,k^{-1})^{k}
         -t^k - \rho(w,w^{-1})^{lw}  \\
   && \;\;\;\;\; + t_L(l)^w + t + \rho(l,w) + \rho(w^{-1}, lw)
         + \rho(k^{-1}, k)).
\end{eqnarray*}
Note that $\rho(k,k^{-1})^k = \rho(k^{-1},k)$ and $\rho(w^{-1}, w l^w) 
+ \rho(w, l^w) = \rho(w^{-1}w, l^w) + \rho(w^{-1}, w)^{l^w}$ hold. Further,
$t-t^k = -[t,k]$. Thus we proved (a).

(b) Follows directly from (a).

(c) If $C_L(\gamma)^g \leq C_H(\sigma)$, then $\zeta_{L,H,g} \in Z^1(L,T)$
by Theorem \ref{morph}. Conversely, let $\zeta_{L,H,g} \in Z^1(L,T)$. Then
$-\gamma = \zeta_{L,H,g}$ and $\sigma : H \ra T : h \ms 0$ satisfy the
conditions of Theorem \ref{morph} and hence $C_L(\gamma)^g \leq C_H(\sigma)$
follows with these choices.

(d) If $g$ centralises $C_L(\gamma)$, then $(l, t_L(l) + \gamma(l)) g 
= g (l, t_L(l) + \gamma(l))$ for each $l \in L$. Expanding both sides
of this equation yields $lw = wl$ and $\Zeta_{L,L,w}(l)^w - \lambda_t(l) 
= \gamma(l) - \gamma(l)^w$ for each $l \in L$. As $\zeta_{L,L,w} = 
\Zeta_{L,L,w} - \lambda_t^w$ by (a), this yields the desired result.
The converse follows with similar arguments.
\end{proof}

\section{A splitting theorem}
\label{split}

Theorem 18 in \cite{ELG08} is a splitting theorem for a certain second
cohomology group. In this section we first generalise this theorem to
cohomology groups in all non-trivial dimensions. For this purpose let 
$P$ be a $p$-group of order $p^m$ and $T \cong \Z_p^d$ a $P$-module. 
Denote the maps induced by projection and inclusion with 
\begin{eqnarray*}
pro_r &:& H^n(P, T) \;\; \ra \;\; H^n(P,T/p^rT), \;\;\; \mbox{ and } \\
inc_{r,m} &:& H^n(P, p^{r-m}T/p^rT) \;\; \ra \;\; H^n(P, T/p^rT).
\end{eqnarray*}

\begin{theorem}
\label{thm18}
Let $P$ be a $p$-group of order $p^m$ and $T \cong \Z_p^d$ a $P$-module. 
For $n \geq 1$ and $r \geq 2m$ it follows that
\begin{eqnarray*}
H^n(P, T/p^rT) \;\;\; = \;\;\; Image(pro_r) &\oplus& Image(inc_{r,m}) \\
               \cong \;\;\;\;\;\;  H^n(P, T) &\oplus& H^{n+1}(P, T),
\end{eqnarray*}
and this splitting is natural with respect to restriction to subgroups 
of $P$ and with respect to the action of $P$ on cohomology groups.
\end{theorem}

\begin{proof}
We prove the claim using essentially the same proof as for Theorem 18
in \cite{ELG08}. The short exact sequence $0 \ra p^r T \ra T \ra T/p^rT
\ra 0$ induces the long exact sequence of cohomology groups
\[ \ldots \ra H^n(P, p^r T) \stackrel{inc_r}{\ra} H^n(P, T)
                            \stackrel{pro_r}{\ra} H^n(P, T/p^rT)
                            \stackrel{con_r}{\ra} H^{n+1}(P, p^rT)
                            \stackrel{inc_r}{\ra} \ldots \]
As $p^r \geq |P|$ and $n \geq 1$, transfer theory implies that $inc_r$  
vanishes, see for example \cite[Proposition 3.6.17]{Benson:I}. As $r-m 
\geq m$, there exists the following commutative diagram with exact rows
\[
\begin{CD}
0 @>>> H^n(P,p^{r-m} T) @>{pro_r}>> H^n (P, p^{r-m} T/p^rT) @>{con_r}>>
                                              H^{n+1}(P,p^rT) @>>> 0 \\
& & @V{0}VV @V{inc_{r,m}}VV @| \\
0 @>>> H^n(P,T) @>{pro_r}>> H^n (P, T/p^rT) @>{con_r}>>
                                              H^{n+1}(P,p^rT) @>>> 0 \\
& & @| @V{pro_{r,m}}VV @V{0}VV \\
0 @>>> H^n(P,T) @>{pro_{r-m}}>> H^n(P, T/p^{r-m}T) @>{con_{r-m}}>>
       H^{n+1}(P,p^{r-m}T) @>>> 0
\end{CD}
\]
Hence $pro_{r,m} \circ con_{r-m} = 0$ and $Image(pro_{r,m}) \subseteq 
Kernel(com_{r-m}) = pro_{r-m}(H^n(P,T))$ follows, since the rows of the 
diagram are
exact. As $pro_{r-m}$ is injective, we can define $pro_{r,m} \circ 
pro_{r-m}^{-1} : H^n(P,T) \ra H^n(P,T/p^rT)$. The kernel $K$ of this map is 
a complement to $pro_r(H^n(P,T))$ in $H^n(P,T/p^rT)$. Further, the map 
$inc_{r-m} : H^n(P,p^{r-m}T) \ra H^n(P,T)$ vanishes. Hence the image of 
$inc_{r,m} : H^n(P, p^{r-m}T/p^rT) \ra H^n(P, T/p^rT)$ coincides with $K$. 
Thus $K = Image(inc_{r,m}) \cong H^{n+1}(P, p^rT) \cong H^{n+1}(P, T)$. 

The splitting is natural with respect to restriction to subgroups,
as $|L| \leq |P|$ for $L \leq P$ and the splitting is natural with 
respect to the action of $P$ as all maps in the diagram are compatible 
with this action.
\end{proof}

Next we consider a variation of Theorem \ref{thm18} for groups of cocycles.
Again, let $P$ be a $p$-group of order $p^m$ and $T \cong \Z_p^d$ a $P$-module
and denote the maps induced by projection and inclusion with 
\begin{eqnarray*}
pro_r &:& Z^n(P, T) \;\; \ra \;\; Z^n(P, T/p^rT) \\
      && \;\;\;\;\; \mbox{ with } \;\; I^n(P, T/p^rT) := Image(pro_r), 
         \;\; \mbox{ and } \\
inc_{r,m} &:& Z^n(P, p^{r-m}T/p^rT) \;\; \ra \;\; Z^n(P, T/p^rT) \\
      && \;\;\;\;\; \mbox{ with } \;\; J^n(P, T/p^rT) := Image(inc_{r,m}),
         \;\; \mbox{ and } \\
inc^*_{r,m} &:& Z^n(P, p^mT/p^rT) \;\; \ra \;\; Z^n(P, T/p^rT) \\
      && \;\;\;\;\; \mbox{ with } \;\; J^{n,*}(P, T/p^rT) 
          := Image(inc^*_{r,m}).
\end{eqnarray*}

\begin{lemma}
\label{lem18}
Let $P$ be a $p$-group with $|P| = p^m$ and let $T \cong \Z_p^d$ a 
$P$-module. Let $n \geq 1$ and $r \geq 2m$ and $p^k = exp(H^n(P, T))$.
\begin{items}
\item[\rm (a)]
There exists a subgroup $K^n(P,T/p^rT) \leq J^n(P, T/p^rT)$ so that
$K^n(P,T/p^rT)$ is a complement to $I^n(P, T/p^rT)$ in $Z^n(P, T/p^rT)$
\item[\rm (b)]
$p^{r-m} B^n(P, T/p^rT) \leq I^n(P, T/p^rT) \cap J^n(P, T/p^rT) \leq 
p^{r-m-k} B^n(P, T/p^rT)$ and the latter is contained in $B^n(P, T/p^rT) 
\leq I^n(P, T/p^rT)$.
\item[\rm (c)]
$B^n(P, T/p^rT) + J^{n,*}(P, T/p^rT) = B^n(P, T/p^rT) + J^n(P, T/p_rT)
= B^n(P, T/p^rT) + K^n(P, T/p^rT)$.
\end{items}
We visualise these results in Figure \ref{fig18}. \begin{figure}[htb]
\begin{center}
\setlength{\unitlength}{3144sp}%
\begingroup\makeatletter\ifx\SetFigFont\undefined%
\gdef\SetFigFont#1#2#3#4#5{%
  \reset@font\fontsize{#1}{#2pt}%
  \fontfamily{#3}\fontseries{#4}\fontshape{#5}%
  \selectfont}%
\fi\endgroup%

\begin{picture}(2730,4121)(2866,-5885)
\put(4051,-2351){\circle*{90}}
\put(4051,-3051){\circle*{90}}
\put(4051,-3701){\circle*{90}}
\put(4051,-4331){\circle*{90}}
\put(4051,-5021){\circle*{90}}
\put(4051,-5731){\circle*{90}}

\put(5401,-5251){\circle*{90}}
\put(5401,-4551){\circle*{90}}
\put(5401,-3901){\circle*{90}}
\put(5401,-3251){\circle*{90}}
\put(5401,-2581){\circle*{90}}
\put(5401,-1861){\circle*{90}}

\put(4051,-2311){\line( 0,-1){3420}}
\put(5401,-1861){\line( 0,-1){3420}}
\put(4051,-5731){\line( 3, 1){1350}}
\put(4087,-5029){\line( 3, 1){1350}}
\put(4087,-4309){\line( 3, 1){1350}}
\put(4087,-3679){\line( 3, 1){1350}}
\put(4087,-3049){\line( 3, 1){1350}}
\put(4087,-2329){\line( 3, 1){1350}}

\put(3601,-5821){\makebox(0,0)[lb]{$\{0\}$}}
\put(3051,-5191){\makebox(0,0)[lb]{$p^{r-m}B^n$}}
\put(3081,-4501){\makebox(0,0)[lb]{$I^n \cap J^n$}}
\put(3601,-3231){\makebox(0,0)[lb]{$B^n$}}
\put(2881,-3931){\makebox(0,0)[lb]{$p^{r-m-k} B^n$}}
\put(3691,-2531){\makebox(0,0)[lb]{$I^n$}}
\put(5541,-4071){\makebox(0,0)[lb]{$J^n$}}
\put(5541,-5401){\makebox(0,0)[lb]{$K^n$}}
\put(5541,-2001){\makebox(0,0)[lb]{$Z^n$}}
\end{picture}%

\end{center}
\caption{Subgroups of $Z^n(P, T/p^rT)$}
\label{fig18}
\end{figure}
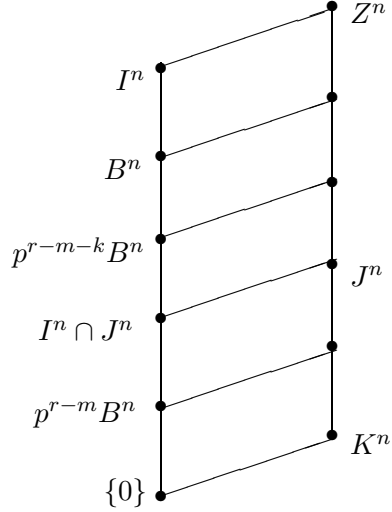

\end{lemma}

\begin{proof}
(a) By definition, the group $Z^n(P,T)$ can be considered as the set
of solutions of a system of linear equations over $\Z_p$. (As $T \cong
\Z_p^d$, there are $d |P|^{n+1}$ equations in $d |P|^n$ unknowns.) Let 
$M$ be the matrix corresponding to this system so that $Z^n(P,T)$ is the 
set of solutions of $Mx = 0$. Then there exist invertible matrices $P$ 
and $Q$ so that $P M Q = D$ is a diagonal matrix with the $t$ diagonal 
entries $(p^{e_1}, \ldots, p^{e_s}, 0, \ldots, 0)$, say. The equation
$Mx = 0$ is equivalent to $Dy = 0$ with $x = Qy$. Hence $Z^n(P,T)$ is 
isomorphic to $\Z_p^{t-s}$. The group $Z^n(P,T/p^rT)$ corresponds to the 
set of solutions of $Mx \equiv 0 \bmod p^r$. This equation is equivalent 
to $Dy \equiv 0 \bmod p^rT$ and thus $Z^n(P,T/p^rT) \cong pro_r(Z^n(P,T)) 
\oplus W_r$ with $W_r$ the finite abelian group with abelian invariants 
$(p^{l_1}, \ldots, p^{l_s})$, where $l_i = \min\{e_i, r\}$. As $p^m 
H^n(P, T/p^rT) = 0$ and $r \geq 2m$, it follows that $e_i \leq r$ and
thus $l_i = e_i$ for $1 \leq i \leq s$. Further, $W_r$ can be chosen
so that it is contained in $J^n(P, T/p^rT)$. Choosing $K^n(P, T/p^rT) 
= W_r$ yields the desired result.

(b) Projection maps $B^n(P,T)$ onto $B^n(P, T/p^rT)$. Thus
$B^n(P, T/p^rT) \leq I^n(P, T/p^rT)$ holds. Next, let $\varphi : 
Z^n(P, T/p^rT) \ra H^n(P, T/p^rT)$ the natural map onto a quotient. 
By Theorem \ref{thm18}, there is a splitting $H^n(P,T/p^rT) = H^n(P,T) 
\oplus H^{n+1}(P,T)$ and $I^n(P,T/p^rT)$ maps onto $H^n(P,T)$ under 
$\varphi$ and $J^n(P,T/p^rT)$ maps onto $H^{n+1}(P,T)$.  This yields that 
$I^n(P,T/p^rT) \cap J^n(P, T/p^rT) \leq ker(\varphi) = B^n(P,T/p^rT)$.
As $p^{r-m} B^n(P, T/p^rT) = B^n(P, p^{r-m} T/p^rT) \leq J^n(P, T/p^rT)$
by construction, it follows that $p^{r-m} B^n(P, T/p^rT) \leq I^n(P, T/p^rT) 
\cap J^n(P, T/p^rT)$. Finally, from the definition of $J^n(P, T/p^rT)$ 
follows that $I^n(P, T/p^rT) \cap J^n(P, T/p^rT) = pro_r( Z^n(P, p^{r-m}T))$ 
and $Z^n(P, p^{r-m}T) = p^{r-m} Z^n(P, T)$. As
$p^k Z^n(P, T) \leq B^n(P, T)$, this yields $p^{r-m} Z^n(P, T) \leq 
p^{r-m-k} B^n(P, T)$. It follows now that $I^n(P, T/p^rT) \cap 
J^n(P, T/p^rT) = pro_r(p^{r-m} Z^n(P, T)) \leq pro_r( p^{r-m-k} B^n(P, T)) 
= p^{r-m-k} pro_r(B^n(P, T)) = p^{r-m-k} B^n(P, T/p^rT)$.

(c) We first note that $m \leq r-m$ and thus $J^n(P, T/p^rT) = 
Z^n(P, p^{m-r}T/p^rT) \leq Z^n(P, p^mT/p^rT) = J^{n,*}(P, T/p^rT)$. 
As in (b), let $\varphi : Z^n(P, T/p^rT) \ra H^n(P, T/p^rT)$ the natural 
map onto a quotient and recall that $H^n(P,T/p^rT) = H^n(P,T) \oplus 
H^{n+1}(P,T)$ by Theorem \ref{thm18}. The the image of $J^{n,*}(P,T/p^rT)$
under $\varphi$ contains the direct summand $H^{n+1}(P, T)$, as this is the 
image of $J^n(P, T/p^rT)$. Let $\alpha \in I^n(P, T/p^rT) \cap J^{n,*}
(P, T/p^rT)$. Then $\alpha = pro_r(\beta)$ and $\beta \in Z^n(P, p^m T)
= p^m Z^n(P, T)$. As $p^m = |P|$, we note that $p^m H^n(P, T) = \{0\}$ and
thus $p^m Z^n(P, T) \leq B^n(P,T)$. Thus $\alpha \in B^2(P, T/p^rT)$.
Hence the image of $J^{n,*}(P, T/p^rT)$ under $\varphi$ is exactly 
$H^{n+1}(P,T)$. And this proves the desired result.
\end{proof}

The complement $K^n(P, T/p^rT)$ as obtained in Lemma \ref{lem18}(a) is
not necessarily unique. The next Lemma asserts that the complements 
$K^n(P,T/p^rT)$ can be chosen so that they are consistent with each 
other. For this purpose note that the multiplication $T/p^rT \ra T/p^{r+1}T : 
t + p^rT \ms pt + p^{r+1}T$ induces a map
\[ mul : Z^n(P, T/p^rT) \ra Z^n(P, T/p^{r+1}T). \]

\begin{lemma} \label{choice}
Let $n \geq 1$, let $P$ be a $p$-group of order $p^m$ and $T \cong \Z_p^d$ 
a $P$-module. Then for each $r \geq 2m$ and each $L \leq P$ we can choose
$K^n(L, T/p^rT)$ so that
\[ mul( K^n(L, T/p^rT)) = K^n(L, T/p^{r+1}T).\]
\end{lemma}

\begin{proof}
For each $n \geq 1$ and each $L \leq P$ choose a complement 
$K^n(L, T/p^{2m}T)$ and then define $K^n(L, T/p^{r+1}T)$ as
$mul( K^n(L, T/p^rT))$. This yields the desired result.
\end{proof}

We note that the subgroups $I^n(P, T/p^rT)$ and $J^n(P, T/p^rT)$ are 
invariant under the conjugation action of $P$ on $Z^n(P, T/p^rT)$, but 
it is in general not possible to find a complement $K^n(P, T/p^rT)$ that 
is also invariant under this conjugation action. The complements 
$K^n(P, T/p^rT)$ are also not necessarily compatible with restriction to 
subgroups $L \leq P$. 

Throughout the remainder of this paper we will assume that we have
chosen $K^n(P, T/p^rT)$ so that Corollary \ref{choice} holds. Then for 
$\gamma \in Z^n(P, T/p^rT)$ let $\ol{\gamma} \in Z^n(P, T)$ and 
$\ul{\gamma} \in K^n(P,T/p^rT)$ so that $\gamma = pro_r( \ol{\gamma} ) 
+ \ul{\gamma}$. Thus we can define an epimorphism
\begin{eqnarray*}
epi = epi_{n,r} &:& Z^n(P, T/p^rT) \ra Z^n(P, T/p^{r-1}T) \\
    &:& pro_r(\ol{\gamma}) + \ul{\gamma} \ms
        pro_{r-1}(\ol{\gamma}) + div(\ul{\gamma}).
\end{eqnarray*}

%%%%%%%%%%%%%%%%%%%%%%%%%%%%%%%%%%%%%%%%%%%%%%%%%%%%%%%%%%%%%%%%%%%%%%%%%%%%%
\section{Coclass families of finite $p$-groups}
\label{coclass}

We first recall the explicit construction of coclass families of 
finite $p$-groups of coclass $r$ from \cite{ELG08} and then 
investigate their Quillen categories. 

%%%%%%%%%%%%%%%%%%%%%%%%%%%%%%%%%%%%%%%%%%%%%%%%%%%%%%%%%%%%%%%%%%%%%%%%%%%%%
\subsection{Coclass families via extensions}

Every infinite coclass family is associated with an infinite pro-$p$-group
$S$ of coclass $r$. As in Section \ref{infinite}, we consider $S$ as an 
extension of $P$ by $T$ via $\rho \in Z^2(P,T)$ and we denote $|P| = p^m$.
Let $e = 3m$ and for $x \in \N_0$ let $M_x = T/p^{x+e}T$. Then $M_x$ is a 
$P$-module via the conjugation of $S$ on $T$. We denote the projection onto 
$M_x$ by $pro_x$ and division by $p^{x+e}$ by $div_x$ to simplify notation.
Theorem \ref{thm18} asserts that $Z^2(P, M_x) = I^2(P, M_x) \oplus 
K^2(P, M_x)$, where $I^2(P, M_x)$ is the image of $Z^2(P,T)$ in $Z^2(P,M_x)$ 
under $pro_x$ and $K^2(P, M_x) \cong H^3(P,T)$ via $div_x \circ con$. Let 
$\rho_x$ be the image of $\rho$ in $Z^2(P,M_x)$ under $pro_x$.
\medskip

{\it Definition:}
For $\eta \in H^3(P, T)$ let $\eta_x$ be the unique preimage in $K^2(P,M_x)$
under $div_x \circ con$. Let $G_x$ be the extension of $P$ by $M_x$ via the 
cocycle $\rho_x + \eta_x$. Then $(G_x \mid x \in \N_0)$ is the coclass 
family defined by $\eta$.
\medskip

The group $G_x$ is the set $P \times M_x$ equipped with the multiplication
\[ (g,m)(h,n) = (gh, m^h + n + \rho_x(g,h) + \eta_x(g,h)).\]
Let $G_x \ra P : (g,m) \ms g$ be the natural epimorphism of $G_x$ onto
$P$ defined by the extension structure of $G_x$. For a subgroup $L \leq P$ 
we denote with $\ol{L}_x$ be the full preimage of $L$ in $G_x$ under this 
natural epimorphism. 

%%%%%%%%%%%%%%%%%%%%%%%%%%%%%%%%%%%%%%%%%%%%%%%%%%%%%%%%%%%%%%%%%%%%%%%%%%%%%
\subsection{The objects of $\A_p(G_x)$}
\label{objectsx}

Let $(G_x \mid x \in \N_0)$ be the coclass family defined by $\eta \in 
H^3(P, T)$. Recall from Section \ref{infinite} that $\L = \{ L \leq P \mid L 
\mbox{ is elementary abelian with } [\rho_L] = 0\}$. Here we define
\[ \L_\eta = \{ L \in \L \mid \eta_L = 0 \}.\]

\begin{lemma} \label{step0}
Let $L$ be an elementary abelian subgroup of $P$ and $x \in \N_0$.
\begin{items}
\item[\rm (a)]
$\ol{L}_x$ splits over $M_x$ if and only if $L \in \L_\eta$.
\item[\rm (b)]
If $\ol{L}_x$ splits over $M_x$, then $\ol{L}$ splits over $T$.
\end{items}
\end{lemma}

\begin{proof}
(a) Suppose that $\ol{L}_x$ splits over $M_x$. Then $res_L(\rho_x + \eta_x) 
= res_L(\rho_x) + res_L(\eta_x) \in B^2(L, M_x)$. Let $|L| = p^l$ and $|P|
= p^m$. By the construction of $\eta_x$ from $\eta$, it follows that the 
image of $\eta_x$ is contained in $p^{x+e-m}T / p^{x+e}T \leq p^mT/p^{x+e}T 
\leq p^l T/p^{x+e}T$. Lemma \ref{lem18}(c) thus implies that $res_L(\eta_x) 
= \alpha + \beta \in B^2(L, M_x) + K^2(L, M_x)$. Further, $res_L(\rho_x) = 
res_L(pro_{x+e}(\rho)) = pro_{x+e}(res_L(\rho)) \in I^2(L, T/p^{x+e}T)$. As 
$K^2(L, T/p^{x+e}T)$ is a complement to $I^2(L, T/p^{x+e}T)$ and 
$B^2(L, T/p^{x+e}T) \leq I^2(L, T/p^{x+e}T)$, it follows that $\beta = 0$. 
Thus $res_L(\eta_x) \in B^2(L, M_x)$ and $res_L(\rho_x) \in B^2(L, M_x)$. 
Hence $res_L(\rho) \in B^2(L, T)$ and $L \in \L$. Further, $res_L(\eta_x) 
\in B^2(L, M_x)$ and $[res_L(\eta_x)] = 0$. This implies that $\eta_L = 0$ 
and thus yields $L \in \L_\eta$. The converse follows with similar arguments.

(b) If $\ol{L}_x$ splits over $M_x$, then (a) implies that $L \in \L$
and thus $\ol{L}$ splits over $T$.
\end{proof}

Let $L \in \L_\eta$. Then as in Lemma \ref{step0}(a), it follows that
$res_L(\eta_x) \in B^2(L, M_x)$. As $\eta_x$ is defined as an element
of $K^n(P, M_x)$, it follows from Lemma \ref{lem18} that $res_L(\eta_x) 
\in B^2(L, M_x) \cap J^n(L, M_x) \leq p^{x+e-l-k} B^2(L, M_x)$ for $l =
\log_p(|L|) \leq m$ and $p^k = exp(H^2(L, T)) \leq p^l \leq p^m$. Hence 
there exists a map 
\[ \omega_{L,x} : L \ra p^{x+e-2m} M_x\] with 
$\eta_x(l,h) = \omega_{L,x}(lh) - \omega_{L,x}(l)^h - \omega_{L,x}(h)$ 
for $l,h \in L$. In other words, $res_L(\eta_x) = -\Delta_1(\omega_{L,x})$. 
Recall that the map $\eta_x$ is chosen so that $\eta_x = div(\eta_{x+1})$ 
for each $x \in \N_0$. Similarly, we choose $\omega_{L,x}$ so that 
$\omega_{L,x} = div(\omega_{L,x+1})$ holds for each $x \in \N_0$.

Again, let $L \in \L_\eta$. Then $C_L = \{(l, t_L(l) \mid l \in L\}$ is a 
complement to $T$ in $\ol{L}$ in the infinite pro-$p$-group $S$. We define
\[ C_{L,x} = \{ (l, pro_x(t_L(l)) + \omega_{L,x}(l)) \mid l \in L\}.\]

\begin{lemma} \label{imageobj}
Let $x \in \N_0$ and $L \in \L_\eta$. Then $C_{L, x}$ is a subgroup
of $G_x$ and hence a complement to $M_x$ in $\ol{L}_x$.
\end{lemma}

\begin{proof}
This follows directly from Section \ref{compl}. We include an outline of the
proof for completeness. As $C_L$ is a subgroup of the infinite group $S$,
we obtain that $(l, t_L(l)) (h, t_L(h)) = (lh, t_L(l)^h + t_L(h) + \rho(l,h)) 
= (lh, t_L(lh))$. Hence $t_L(l)^h + t_L(h) + \rho(l,h) = t_L(lh)$. Similarly,
in $G_x$ we obtain
\begin{eqnarray*}
 &&  (l, pro_x(t_L(l))+\omega_{L,x}(l)) (h, pro_x(t_L(h))+\omega_{L,x}(h)) \\
 &=& (lh, pro_x(t_L(l))^h+pro_x(t_L(h))+\omega_{L,x}(l)^h + \omega_{L,x}(h) +
      \rho_x(l,h) + \eta_x(l,h)) \\
 &=& (lh, pro_x(t_L(lh)) + \omega_{L,x}(lh) )
\end{eqnarray*}
Hence $C_{L, x}$ is a subgroup of $G_x$. By construction, $C_{L, x} \cap 
M_x = \{1\}$ and $C_{L, x} \cong L$ via the natural epimorphism $G_x \ra P$. 
Thus $C_{L,x}$ is a complement to $M_x$ in $\ol{L}_x$.
\end{proof}

For $L \in \L_\eta$ we denote $t_{L,x}(l) = pro_x(t_L(l)) + \omega_{L,x}(l) \in 
M_x$ and $c_{L,x}(l) = (l, t_{L,x}(l)) \in G_x$. Then $C_{L,x} = \{ c_{L,x}(l)
\mid l \in L \}$. Further, for $\gamma \in Z^1(L, M_x)$ we write 
\[ C_{L,x}(\gamma) 
   = \{ (l, t_{L,x}(l) + \gamma(l)) \mid l \in L \}.\]

As $C_{L,x}$ is a complement to $M_x$ in $\ol{L}_x$ by Lemma \ref{imageobj},
it follows that $C_{L,x}(\gamma)$ is also a complement to $M_x$ in $\ol{L}_x$.

Let $\O_x(L)$ denote the set of elementary abelian subgroups of $M_x$ 
which are centralised by $L$. As $M_x$ is abelian of rank $d$, there 
exists a unique maximal elementary abelian subgroup $A_x = p^{x+e-1} M_x$ 
in $M_x$ and this has rank $d$ as well. The subgroups in $\O_x(L)$ are 
the subgroups of the centraliser of $L$ in $A_x$. If $O \in \O_x(L)$, 
then $O$ is centralised by $L$ and thus also by $\ol{L}_x$ and by 
$C_{L,x}(\gamma)$. Hence the subgroup of $G_x$ generated by $O$ and 
$C_{L,x}(\gamma)$ is a direct product $C_{L,x}(\gamma) \times O$. We define

\[ \C_x(L) = \{ C_{L, x}(\gamma) \times O
   \mid \gamma \in Z^1(L,M_x) \mbox{ and } O \in \O_x(L) \}. \]

\begin{theorem}
The objects in the Quillen category $\A_p(G_x)$ are the disjoint union
of the sets $\C_x(L)$ for $L \in \L_\eta$.
\end{theorem}

\begin{proof}
Let $E$ be an object in $\A_p(G_x)$. Let $O = E \cap M_x$ and $L = E M_x/M_x
\leq P$. As $E$ is elementary abelian, there exists a complement $F$ to $O$
in $E$. It follows that $F$ is a complement to $M_x$ in $\ol{L}_x$ and hence 
$L \in \L_\eta$ by Lemma \ref{step0}. Further, we obtain that 
$F = C_{L,x}(\gamma)$ for some $\gamma \in Z^1(L, M_x)$, since every 
complement to $M_x$ in $\ol{L}_x$ is of this form. Finally, 
$E = F \times O$. Thus $E \in \C_x(L)$. The converse is obvious.
\end{proof}

Thus each object $E$ in $\A_p(G_x)$ has the form $C_{L,x}(\gamma) \times O$.
The subgroups $L$ and $O$ are uniquely defined by $E$. We note that different 
elements $\gamma, \delta \in Z^1(L, M_x)$ can define the same object 
$C_{L,x}(\gamma) \times O = C_{L,x}(\delta) \times O$ and thus the description
of the objects in $C_x(L)$ is not irredundant.

%%%%%%%%%%%%%%%%%%%%%%%%%%%%%%%%%%%%%%%%%%%%%%%%%%%%%%%%%%%%%%%%%%%%%%%%%%%%%
\subsection{The morphisms of $\A_p(G_x)$}
\label{morphismsx}

In this subsection we describe the morphisms of $\A_p(G_x)$. As above, 
let $(G_x \mid x \in \N_0)$ be the coclass family defined by $\eta \in
H^3(P,T)$. We investigate under which conditions an object of $\A_p(G_x)$
conjugates under $g \in G_x$ into another object.

For $L, H \in \L_\eta$ and $g \in G_x$ with $L^g \leq H$ we define
\[ \zeta_{L,H,g,x} : L \ra G_x : c_{H,x}(l^g)^{-g^{-1}} \cdot c_{L,x}(l) \]
and for $\sigma \in Z^1(H,M_x)$ we denote
\[ \sigma_L^g : L \ra T : l \ms \sigma(l^g)^{g^{-1}} \in Z^1(L, M_x).\]

\begin{theorem} \label{morphx}
Let $A = C_{L,x}(\gamma) \times O$ and $B = C_{H,x}(\sigma) \times U$ be 
two objects in $\A_p(G_x)$ and let $g \in G_x$. Then $A^g \leq B$ if and 
only if $L^g \leq H$, and $O^g \leq U$, and there exists $\delta \in 
Z^1(L,M_x)$ with image in $U^{g^{-1}}$ so that $\sigma_L^g - \gamma + 
\delta = \zeta_{L,H,g,x}$.
\end{theorem}

\begin{proof}
First suppose that $A^g \leq B$. As $L = A M_x/M_x$ and $H = B M_x / M_x$,
it follows that $L^g \leq H$. As $O = A \cap M_x$ and $U = B \cap M_x$,
it follows that $O^g \leq U$. It thus remains to show the existence of
$\delta$. 

Recall that $C_{H,x}(\sigma)$ is a complement to $M_x$ in $\ol{H}_x$. 
Let $\epsilon : G_x \ra P$ be the natural homomorphism induced by the 
extension structure of $G_x$. Then restricting $\epsilon$ to subgroups
induces an isomorphism $C_{H,x}(\sigma) \ra H$. Let $W$ be the preimage 
of $L^g$ under this isomorphism. Then $W \times U$ is the full preimage 
of $L^g$ under the restriction of $\epsilon$ to $B$.

Note that $(C_{L,x}(\gamma))^g$ is contained in $B$ and is mapped onto
$L^g$ under $\epsilon$. Thus $(C_{L,x}(\gamma))^g$ is contained in the 
full preimage of $L^g$ under the restriction of $\epsilon$ to $B$. Hence
$(C_{L,x}(\gamma))^g \leq W \times U$ follows. As $C_{L,x}(\gamma)$ is
a complement to $M_x$ in $\ol{L}_x$, it follows that $(C_{L,x}(\gamma))^g$
intersects trivially with $U$. Thus $(C_{L,x}(\gamma))^g$ is a complement
to $U$ in $W \times U$. 

The elements of $(C_{L,x}(\gamma))^g$ are of the form 
$(c_{L,x}(l) (1, \gamma(l)))^g = (c_{L,x}(l))^g (1, \gamma(l)^g)$ and
this element maps onto $l^g$ under $\epsilon$. The elements of $W$ are
of the form $c_H(l^g) (1, \sigma(l^g))$ and this element maps onto $l^g$
under $\epsilon$ as well. As $(C_{L,x}(\gamma))^g$ and $W$ are both
complements to $U$ in $W \times U$ and $W \cong L^g$, there exists
$\delta' \in Z^1(L^g, U)$ with 
$(c_{L,x}(l))^g (1, \gamma(l)^g) 
  = c_{H,x}(l^g) (1, \sigma(l^g) + \delta'(l^g))$.
This translates to 
$c_{H,x}(l^g)^{-1} (c_{L,x}(l))^g 
  = (1, \sigma(l^g) + \delta'(l^g) - \gamma(l)^g)$.
Conjugating with $g^{-1}$ now produces
$\zeta_{L,H,g,x}(l) = c_{H,x}(l^g)^{-g^{-1}} c_{L,x}(l) 
  = (1, \sigma(l^g)^{g^{-1}} + \delta'(l^g)^{g^{-1}} - \gamma(l))$.
Defining $\delta : L \ra U^{g^{-1}} : l \ms \delta'(l^g)^{g^{-1}}$
yields that $\sigma_L^g - \gamma + \delta = \zeta_{L,H,g,x}$ as
desired.

The converse follows with similar arguments.
\end{proof}

We investigate $\zeta_{L,H,g,x}$ in more detail using the extension structure 
of $G_x$. Recall that $G_x$ is an extension of $P$ by $M$ via $\nu_x := 
\rho_x + \eta_x \in Z^2(P,M_x)$. For $L, H \in \L_\eta$ and $w \in P$ let
\begin{eqnarray*}
\varphi_{L, H, w,x} &:& L \ra M_x : 
    l \ms (\omega_{L,x}(l)^w - \omega_{H,x}(l^w) 
             + \eta_x(l, w) - \eta_x(w,l^w) )^{w^{-1}} \\
\Zeta_{L, H, w, x} &:& L \ra M_x : 
    l \ms (t_{L,x}(l)^w -t_{H,x}(l^w) 
             + \nu_x(l, w) - \nu_x(w, l^w))^{w^{-1}}.
\end{eqnarray*}
For $m \in M_x$ let $\lambda_m : L \ra M_x : l \ms [m, l] \in B^1(L, M_x)$
and note that $\lambda_m^w : L \ra M_x : l \ms [m,l^w]^{w^{-1}} = 
[m^{w^{-1}}, l] \in B^1(L,M_x)$. The following lemma is dual to Lemma 
\ref{zeta} and can be proved with similar arguments.

\begin{lemma} \label{zetax}
Let $L, H \in \L_\eta$ and $g = (w,m) \in G_x$ with $L^g \leq H$. Let $h 
= (w,t) \in S$ with $pro_x(t) = m$. Then 
\begin{items}
\item[\rm (a)]
$\zeta_{L,H,g,x} = \Zeta_{L,H,w,x} - \lambda_m^w$.
\item[\rm (b)]
if $\zeta_{L,H,g,x} \in Z^1(L,M_x)$, then $\zeta_{L,H,g',x} \in Z^1(L,M_x)$
for all $g' \in gM_x$.
\item[\rm (c)]
$\zeta_{L,H,g,x} = pro_x(\zeta_{L,H,h}) + \varphi_{L,H,w,x}$ and 
$\Zeta_{L,H,w,x} = pro_x(\Zeta_{L,H,w}) + \varphi_{L,H,w,x}$.
\item[\rm (d)]
$g$ centralises $C_{L,x}(\gamma)$ if and only if $w$ centralises $L$ and
$\zeta_{L,L,g,x} = \gamma^g - \gamma$.
\end{items}
\end{lemma}

We investigate in more detail how $\Zeta_{L,H,w}$ and $\Zeta_{L,H,w,x}$
are connected to each other.

\begin{lemma} \label{prozeta}
There exists $x_0 \in \N$ so that for all $L, H \in \L_\eta$ and all
$w \in P$ we obtain that
\[ \Zeta_{L,H,w,x} \in Z^1(L, M_x) \mbox{ for some } x \geq x_0
\;\;\; \Ra \;\;\; \Zeta_{L,H,w} \in Z^1(L, T). \] 
\end{lemma}

\begin{proof}
Let $L, H \in \L_\eta$ and $w \in P$. For $l_1, l_2 \in L$ denote
$\delta(l_1, l_2) = \Zeta_{L,H,w}(l_1l_2) - \Zeta_{L,H,w}(l_1)^{l_2}
   - \Zeta_{L,H,w}(l_2)$.
If $\Zeta_{L,H,w} \not \in Z^1(L, T)$, then $\delta \neq 0$ and thus 
there exists $x \in \N$ with $pro_x(\delta) \neq 0$. Let $x_{L,H,w}$ 
be the minimal $x \in \N$ with $pro_x(\delta) \neq 0$ in this case. 
If $\Zeta_{L,H,w} \in Z^1(L, T)$, then set $x_{L,H,w} = 0$.

Let $x_{max}$ be the maximum of the finitely many values $x_{L,H,w}$ for
$L, H \in \L_\eta$ and $w \in P$ and define $x_0 = x_{max} + m$ for $m =
\log_p(|P|)$. Let $x \geq x_0$. Now consider $\Zeta_{L,H,w,x} \in 
Z^1(L, M_x)$ and recall that $\varphi_{L,H,w,x}$ has its image in 
$p^{x+e-m}M_x$. Thus $pro_{x-m}(\Zeta_{L,H,w}) = pro_{x-m}(\Zeta_{L,H,w,x}) 
\in Z^1(L, pro_{x-m}(M_x))$.  As $x-m \geq x_{max} \geq x_{L,H,w}$, it now 
follows that $\Zeta_{L,H,w} \in Z^1(L, T)$ by the construction of $x_0$.
\end{proof}

As a final step of this subsection we compare different elements
inducing the same morphism in $\A_p(G_x)$.

\begin{lemma} \label{samemor}
Let $g = (w,m)$ and $g' = (w, m')$ be two elements of $G_x$. Then $g$
and $g'$ induce the same morphism in $Hom(A,B)$ if and only if 
$m - m' \in C_{M_x}(L)^w = \{ n \in M_x \mid [n,l] = 1 \mbox{ for } 
l \in L\}^w$.
\end{lemma}

\begin{proof}
If $g$ and $g'$ induce the same morphism, then $(c_{L,x}(l))^g =
(c_{L,x}(l))^{g'}$ for each $l \in L$. Thus $g$ and $g'$ induce
the same element $\delta$ in the notation of Theorem \ref{morphx}
and its proof. Additionally using Lemma \ref{zetax} this implies
that 
\begin{eqnarray*}
\lambda_m^w &=& \Zeta_{L,H,w,x} - \zeta_{L,H,g,x} \\
            &=& \Zeta_{L,H,w,x} - (\sigma_L^w - \gamma + \delta) \\
            &=& \Zeta_{L,H,w,x} - \zeta_{L,H,g',x} \\
            &=& \lambda_{m'}^w
\end{eqnarray*}
Hence $m \equiv m' \bmod C_{M_x}(L)^w$. The converse follows with a
similar computation.
\end{proof}

\section{Skeletons of Quillen categories}
\label{skels}

Recall that a skeleton of a category is the full subcategory obtained
by choosing one object from each isomorphism class of objects. A 
semi-skeleton of a category is the full subcategory obtained by choosing
at least one object from each isomorphism class of objects. It is
well-known in category theory that each category is equivalent to each
of its (semi-)skeletons.

%%%%%%%%%%%%%%%%%%%%%%%%%%%%%%%%%%%%%%%%%%%%%%%%%%%%%%%%%%%%%%%%%%%%%%%%%%%%%
\subsection{A semi-skeleton for $\A_p(S)$}

The objects of $\A_p(S)$ have the form $C_L(\gamma)$ for $L \in \L$
and $\gamma \in Z^1(L, T)$. Let $T^1(L, T)$ denote a transversal of 
$B^1(L, T)$ in $Z^1(L, T)$ and denote $\cT = \{ C_L(\gamma) \mid 
\gamma \in T^1(L,T) \}$. Then we define $\ol{\A}_p(S)$ as the full 
subcategory of $\A_p(S)$ on the objects in $\cT$. 

\begin{lemma}
$\ol{\A}_p(S)$ is a semi-skeleton of $S$. Its number of objects is
\[ \sum_{L \in \L} |H^1(L,T)|. \]
\end{lemma}

\begin{proof}
Two objects $C_L(\gamma)$ and $C_H(\sigma)$ are conjugate under the action 
of $T$ if and only if $L = H$ and $\gamma \equiv \sigma \bmod B^1(L, T)$. 
Thus $\ol{\A}_p(S)$ is a semi-skeleton of $\A_p(S)$. It now remains to note
that $|T^1(L,T)| = |H^1(L,T)|$.
\end{proof}

This has the following immediate consequence.

\begin{corollary}
Each skeleton of $\A_p(S)$ is finite.
\end{corollary}

%%%%%%%%%%%%%%%%%%%%%%%%%%%%%%%%%%%%%%%%%%%%%%%%%%%%%%%%%%%%%%%%%%%%%%%%%%%%%
\subsection{A semi-skeleton for $\A_p(G_x)$}

Let $x \in \N$.
The objects of $\A_p(G_x)$ all have the form $C_{L,x}(\gamma) \times O$
for $L \in \L_\eta$, $\gamma \in Z^1(L, M_x)$ and $O \in \O_x(L)$. Let
$W^1(L, M_x) = pro_x(T^1(L,T)) + K^1(L,T)$ and denote $\cT_x = \{
C_{L,x}(\gamma) \times O \mid \gamma \in W^1(L,M_x), O \in \O_x(L) \}$.
Then we define $\ol{\A}_p(G_x)$ as the full subcategory of $\A_p(G_x)$
on the objects in $\cT_x$.

\begin{lemma}
$\ol{\A}_p(G_x)$ is a semi-skeleton of $S$. Its number of objects is
\[ \sum_{L \in \L_\eta} |H^1(L,T)| |H^2(L, T)| |\O_x(L)|. \]
\end{lemma}

\begin{proof}
Two objects $C_{L,x}(\gamma) \times O$ and $C_{H,x}(\sigma) \times U$ are 
conjugate under the action of $M_x$ if and only if $L = H$ and $O = U$ and
$\gamma \equiv \sigma \bmod B^1(L, M_x)$. By Lemma \ref{lem18}, the set
$W^1(L, M_x)$ is a transversal of $B^1(L, M_x)$ in $Z^1(L, M_x)$. Thus 
$\ol{\A}_p(G_x)$ is a semi-skeleton of $\A_p(G_x)$. Further, $|W^1(L, M_x)| 
= |T^1(L, T)| \cdot |K^1(L, M_x)| = |H^1(L, T)| |H^2(L, T)|$. This yields
the desired result.
\end{proof}

This has the following immediate consequence.

\begin{corollary}
The number of objects in $\ol{\A}_p(G_x)$ is independent of $x$.
\end{corollary}

%%%%%%%%%%%%%%%%%%%%%%%%%%%%%%%%%%%%%%%%%%%%%%%%%%%%%%%%%%%%%%%%%%%%%%%%%%%%%
\section{A functor from $\A_p(G_x)$ to $\A_p(S)$}

Let $x \geq x_0$ for $x_0$ as in Lemma \ref{prozeta}. Our aim in this
section is to define a functor from the category $\A_p(G_x)$ to the
category $\A_p(S)$. For an arbitrary category let $Hom(A,B)$ denote the 
set of all morphisms from an object $A$ to an object $B$ and let $\mu_g$ 
denote the morphism induced by conjugation with the element $g$. 

An object $A$ in $\A_p(G_x)$ has the form $A = C_{L,x}(\gamma) \times O$. 
We define $\ol{A} = C_L(\ol{\gamma})$ and thus obtain an object $\ol{A}$ 
in $\A_p(S)$. The following theorem provides the basis for the construction
of morphisms in $\A_p(S)$ from morphisms in $\A_p(G_x)$.

\begin{theorem} \label{shift}
Let $x \geq x_0$ and let $A = C_{L,x}(\gamma) \times O$ and 
$B = C_{H,x}(\sigma) \times U$ be two objects in $\A_p(G_x)$. 
If $g = (w,m) \in G_x$ induces a morphism in $Hom(A,B)$, then 
there exists $t \in T$ so that $h = (w, t) \in S$ induces a morphism 
in $Hom(\ol{A}, \ol{B})$.
\end{theorem}

\begin{proof}
By Theorem \ref{morphx}, the assumption $A^g \leq B$ is equivalent to
$L^w \leq H$, and $O^w \leq U$ and there exists $\delta \in Z^1(L, 
U^{w^{-1}})$ with $\zeta_{L,H,g,x} = \sigma_L^w - \gamma + \delta$.
First note that $\sigma = pro_x(\ol{\sigma}) + \ul{\sigma}$ implies 
that $\sigma^w_L = pro_x(\ol{\sigma}^w_L) + \ul{\sigma}_L^w$, as 
addition and $pro_x$ are compatible with restriction and conjugation.
Let $s \in T$ with $pro_x(s) = m$. Then 
\begin{eqnarray*}
0 
&=& \zeta_{L,H,g,x} - \sigma_L^w + \gamma - \delta \\
&=& \Zeta_{L,H,w,x} - \lambda_m^w - \sigma_L^w + \gamma - \delta \\
&=& pro_x(\Zeta_{L,H,w}) + \varphi_{L,H,w,x} - pro_x(\lambda_s^w) 
    - \sigma_L^w + \gamma - \delta \\
&=& pro_x( \Zeta_{L,H,w} - \lambda_s^w - \ol{\sigma}^w_L + \ol{\gamma})  
   + \varphi_{L,H,w,x} - \ul{\sigma}^w_L + \ul{\gamma} - \delta 
\end{eqnarray*}
and thus
\[ pro_x( \Zeta_{L,H,w} - \lambda_s^w - \ol{\sigma}^w_L + \ol{\gamma})  
  = - \varphi_{L,H,w,x} + \ul{\sigma}^w_L - \ul{\gamma} + \delta. 
  \;\;\;\;\;\;\;\; (*)\]

The map $\Zeta_{L,H,w}$ is an element of $Z^1(L, T)$ by Lemma \ref{prozeta}.
Hence the left hand side of equation (*) is an element of $I^1(L, M_x)$.

We investigate the right hand side of equation (*). First note that 
$\varphi_{L,H,w,x} \in Z^1(L, M_x)$ by Lemma \ref{zetax}(b) and thus the 
right hand side of (*) defines an element of $Z^1(L, M_x)$. By construction, 
the images of the maps $\ul{\gamma}$, $\ul{\sigma}^w_L$ and $\delta$ are 
contained in $p^{x+e-m} M_x$ for $m = \log_p(|P|)$. It remains to consider 
$\varphi_{L,H,w,x}$. The definition of $\varphi_{L,H,w,x}$ is based on 
$\eta_x$, $\omega_{L,x}$ and $\omega_{H,x}$. The map $\eta_x$ has its 
image in $p^{x+e-m} M_x$ by construction and the maps $\omega_{L,x}$ and
$\omega_{H,x}$ are defined such that their images are contained in
$p^{x+e-2m} M_x$. As $e = 3m$, it follows that $x+e-2m \geq m$ and hence 
the right hand side of (*) is contained in $J^{1,*}(L, M_x) \leq 
B^1(L, M_x) + K^1(L, M_x)$ by Lemma \ref{lem18} (c).  

Combining the results on the left and the right hand side of (*), we obtain
that both sides are contained in $B^1(L, M_x) = pro_x(B^1(L,T))$. Thus there 
exists $\ol{s} \in T$ so that 
$pro_x( \Zeta_{L,H,w} - \lambda_{s^{w^{-1}} + \ol{s}} - \ol{\sigma}^w_L 
+ \ol{\gamma}) = 0$. As $ker(pro_x) \leq B^1(L, T)$, there exists $u \in T$ 
with $\Zeta_{L,H,w} - \lambda_{s^{w^{-1}} + \ol{s}} - \ol{\sigma}^w_L + 
\ol{\gamma} = \lambda_u$. Choosing $t = s + (\ol{s} + u)^w$ now yields the 
desired result by Theorem \ref{morph} and Lemma \ref{zeta}. 
\end{proof}

We call an element $h$ as determined in Theorem \ref{shift} a {\em lifting}
of $g$ with respect to the objects $A$ and $B$ in $\A_p(G_x)$. The proof
of Theorem \ref{shift} also yields a way to identify liftings:
Given $g = (w, m)$ together with objects $A = C_{L,x}(\gamma) \times O$ and 
$B = C_{H,x}(\sigma) \times U$, the proof of Theorem~\ref{shift} constructs 
an element $\ol{m} = t \in T$ satisfying
\begin{eqnarray*}
(1) && \Zeta_{L,H,w} - \ol{\sigma}^w_L + \ol{\gamma}
        = \lambda_{\ol{m}}^w \, .
\end{eqnarray*}
Conversely let $\ol{m} \in T$ be any element satisfying (1)\@. Then with 
$\ul{m} := m - pro_x(\ol{m}) \in M_x$ and $\delta \in Z^1(L, M_x)$ as 
defined as in Theorem \ref{morphx}, the proof of Theorem~\ref{shift} shows 
that
\begin{eqnarray*}
(2) && \varphi_{L,H,w,x} - \ul{\sigma}^w_L + \ul{\gamma} - \delta 
        = \lambda_{\ul{m}}^w
\end{eqnarray*}
holds. So $m = pro_x(\ol{m}) + \ul{m}$, $g = (w, pro_x(\ol{m}) + \ul{m})$, 
and $h = (w, \ol{m})$ is a lifting of $g$. By Theorem \ref{morph} we observe 
that every $h = (w,t) \in S$ with $\ol{A}^h \leq \ol{B}$ satisfies equation 
(1). Hence equation (1) characterises the liftings of an element $g$ with 
respect to $A$ and $B$. The next lemmata investigate properties of liftings.

\begin{remark} \label{unique}
Let $x \geq x_0$ and suppose that $g = (w,m) \in G_x$ induces a morphism in 
$Hom(A,B)$. Let $h = (w, \ol{m})$ be a lifting of $g$ with respect to $A$ 
and $B$.
\begin{items}
\item[\rm (a)]
$\{ (w, t) \mid t \in \ol{m}+C_T(L)^w \}$ is the set of all liftings of $g$.
\item[\rm (b)]
All liftings of $g$ define the same morphism in $Hom(\ol{A}, \ol{B})$.
\end{items}
\end{remark}

\begin{proof}
This follows directly from equation (1).
\end{proof}

\begin{lemma} \label{functorial}
Let $x \geq x_0$. Let $g_1, g_2 \in G_x$ so that $g_1$ induces a morphism 
in $Hom(A,B)$ and $g_2$ induces a morphism in $Hom(B,C)$. Let $h_1$ be a 
lifting of $g_1$ with respect to $A$ and $B$ and let $h_2$ be a lifting of 
$g_2$ with respect to $B$ and $C$. Then $h_1 h_2$ is a lifting of $g_1 g_2$ 
with respect to $A$ and $C$.
\end{lemma}

\begin{proof}
Denote $g_i = (w_i, m_i)$ and $h_i = (w_i, t_i)$ for $i = 1,2$. Then 
$g_1g_2 = (w_1w_2,m)$ and $h_1h_2 = (w_1w_2,t)$ for some $m \in M_x$ 
and $t \in T$. Since $A^{g_1} \leq B$ and $h_1$ is a lifting of $g_1$, 
we have $\ol{A}^{h_1} \leq \ol{B}$. Since $B^{g_2} \leq C$ and $h_2$ 
is a lifting of $g_2$, we have $\ol{B}^{h_2} \leq \ol{C}$. So $A^{g_1g_2} 
\leq C$ and $\ol{A}^{h_1h_2} \leq \ol{C}$. Hence $h_1h_2$ is a lifting 
of $g_1g_2$.
\end{proof}

\begin{lemma} \label{identity}
Let $x \geq x_0$. Suppose that $g$ induces the identity morphism in $Hom(A,B)$
and let $h$ be a lifting of $g$ with respect to $A$ and $B$. Then $h$ induces
the identity morphism in $Hom(\ol{A}, \ol{B})$.
\end{lemma}

\begin{proof}
The element $g$ induces the identity morphism in $Hom(A,B)$ if and only
if $g$ centralises $A$. Let $A = C_{L,x}(\gamma) \times O$ and $g = (w,m)$. 
If $g$ centralises $A$, then $g$ centralises $C_{L,x}(\gamma)$ and, by Lemma 
\ref{zetax}(d), it follows that $w$ centralises $L$ and $\zeta_{L,L,g,x} = 
\gamma^g - \gamma$. As in equation (1), it follows that $g$ has a lifting 
$h'$ such that $\zeta_{L,L,h'} = \ol{\gamma}^{h'} - \ol{\gamma}$. By Lemma 
\ref{zeta}(d) we obtain that $h'$ centralises $C_L(\ol{\gamma})$. So $h$ 
does too, by Remark~\ref{unique}(c).  Further $A \leq B$ implies that 
$\ol{A} \leq \ol{B}$ and thus we obtain that $h$ induces the trivial 
morphism in $Hom(\ol{A}, \ol{B})$.
\end{proof}

\begin{lemma} \label{well-def}
Let $x \geq x_0$. Suppose that the elements $g_1$ and $g_2$ of $G_x$ 
induce the same morphism in 
$Hom(A,B)$ and let $h_i$ be a lifting of $g_i$ with respect to $A$ and $B$
for $i = 1,2$. Then $h_1$ and $h_2$ induce the same morphism in $Hom(\ol{A}, 
\ol{B})$.
\end{lemma}

\begin{proof}
By assumption $g_2 = c  g_1$, where $c := g_2 g_1^{-1}$ lies in $C_{G_x}(A)$. 
Let $h$ be lifting of $c$ with respect to $A$ and $A$. Then $hh_1$ is another 
lifting of $g_2$, by Lemma~\ref{functorial}. Hence by Remark~\ref{unique}(c), 
$h_2$ induces the same morphism as $hh_1$. But $h$ induces the identity 
morphism by Lemma~\ref{identity}.
\end{proof}

Based on the definitions of $\ol{A}$ for objects $A$ and liftings $h$ for
elements $g$ we define 
\[ F : \A_p(G_x) \ra \A_p(S)
     : \left \{ \begin{array}[h]{ll}
        A \ms \ol{A} & \mbox{ for objects } A  \mbox{ in } \A_p(G_x) \\
        \mu_g \ms \mu_h & \mbox{ for } g \in G_x \mbox{ with lifting } h \\
       \end{array} \right..\]

Lemma \ref{well-def} shows that this map is well-defined and Lemmas 
\ref{functorial} and \ref{identity} yield that it is a functor. Note 
that $F_S$ maps the semi-skeleton $\ol{A}_p(G_x)$ to the semi-skeleton 
$\ol{\A}_p(S)$.

\section{A functor from $\A_p(G_x)$ to $\A_p(G_{x+1})$}
\label{equiv}

Let $x \geq x_0$ for $x_0$ as in Lemma \ref{prozeta}. Our aim in this
section is to define a functor from the semi-skeleton $\ol{\A}_p(G_x)$
to the semi-skeleton $\ol{\A}_p(G_{x+1})$ and then to show that this 
functor induces an equivalence of categories.

As preliminary step recall that $mul : M_x \ra M_{x+1}: t+p^{x+e}T \ra 
pt + p^{x+e+1}T$ is induced by multiplication with $p$. This also
defines a map $mul$ on cohomology groups and this, in turn, yields 
an isomorphism $K^1(L, M_x) \ra K^1(L, M_{x+1})$.

Let $A = C_{L,x}(\gamma) \times O$ be an object in $\ol{\A}_p(G_x)$. Then 
$\gamma = pro_x(\ol{\gamma}) + \ul{\gamma}$ with $\ol{\gamma} \in T^1(L, T)$ 
and $\ul{\gamma} \in K^1(L, M_x)$. We define $iso : W^1(L, M_x) \ra 
W^1(L, M_{x+1}) : \gamma \ms pro_{x+1}(\ol{\gamma}) + mul(\ul{\gamma})$ and 
note that this is a bijection. This allows to define the following object
in $\ol{\A}_p(G_{x+1})$
\[ \hat{A} = C_{L,x+1}(iso(\gamma)) \times mul(O). \]

Suppose that $g \in G_x$ induces a morphism in $Hom(A,B)$. Write $g = 
(w, pro_x(\ol{m}) + \ul{m})$, where $h = (w, \ol{m})$ is a lifting of $g$. 
Then we define
\[ \hat{g} = (w, pro_{x+1}(\ol{m}) + mul(\ul{m})) .\]
We say that $\hat{g}$ is a {\em pushout} of $g$. Note that $\hat{g}$ depends 
on the choices made for a lifting $h$.

\begin{lemma}
\label{g-hat}
Let $g \in G_x$ and let $h \in S$ be a lifting of $g$ with respect to 
$A$ and $B$. 
\begin{items}
\item[\rm (a)] 
The map $\mu_{\hat{g}}$ is a morphism in $Hom(\hat{A}, \hat{B})$.
\item[\rm (b)]
The morphism $\mu_{\hat{g}}$ is independent of the choices made in the 
definition of $\hat{g}$.
\item[\rm (c)]
The element $h \in S$ is a lifting of $\hat{g}$ with respect to 
$\hat{A}$ and $\hat{B}$.
\end{items}
\end{lemma}

\begin{proof}
(a) and (c) follow readily from Theorem \ref{morphx} and equations (1) and
(2) which characterise the lifting $h$. It remains to consider (b).
Let $h = (w,\ol{m})$ and $h' = (w,\ol{m}')$ be two liftings of $g = (w,m)$, 
corresponding to the decompositions $g = pro_x(\ol{m}) + \ul{m}$ and $g 
= pro_x(\ol{m}') + \ul{m}'$. By Remark \ref{unique}, the element 
$c := \ol{m}'-\ol{m}$ lies in $C_T(L)^w$. Then $\ol{m'} = \ol{m} + c$ and 
$\ul{m'} = \ul{m'} - pro_x(c)$, and so since $mul(pro_x(c)) = pro_{x+1}(pc)$, 
it follows that the pushout $\hat{g}'$ obtained using $h'$ is
\[
\hat{g}' = (w,pro_{x+1} (\ol{m}') + mul(\ul{m}')) 
         = \hat{g}(1,pro_{x+1}((1-p)c)) \, .
\]
Since $c \in C_T(L)^w$ we have $pro_{x+1}((1-p)c) \in C_{M_{x+1}}(L)^w$. 
By Lemma \ref{samemor}, it follows that $\hat{g}'$ and $\hat{g}$ induce 
the same morphism in $Hom(\hat{A}, \hat{B})$.
\end{proof}

Based on the definitions of $\hat{A}$ and $\hat{g}$ we define a map 
\[ F : \ol{\A}_p(G_x) \ra \ol{\A}_p(G_{x+1})
     : \left \{ \begin{array}[h]{ll} 
        A \ms \hat{A} & \mbox{ for } A \in \cT_x \\
        \mu_g \ms \mu_{\hat{g}} & \mbox{ for } g \in G_x \\
       \end{array} \right..\]

\begin{theorem}
Let $x \geq x_0$. Then $F$ is a functor from $\ol{\A}_p(G_x)$ to 
$\ol{\A}_p(G_{x+1})$.
\end{theorem}

\begin{proof}
$F$ maps the objects of $\ol{\A}_p(G_x)$ one-to-one to objects in the 
Quillen category $\ol{\A}_p(G_{x+1})$ by construction. It remains to 
consider the morphisms.  

The map $g \mapsto \mu_{\hat{g}}$ is well-defined by Lemma~\ref{g-hat}(b). 
We now show that $\mu_{\hat{g_1}} \mu_{\hat{g_2}} = \mu_{\widehat{g_1g_2}}$. 
Let $g_i = (w_i,m_i)$ with lifting $h_i = (w_i,\ol{m}_i)$ ($i=1,2$). So 
$g_1g_2 = (w_1w_2,m)$ for
\[
m = m_1^{w_2} + m_2 + \nu_x(w_1,w_2) \, .
\]
By Lemma \ref{functorial}, $h_1 h_2$ is a lifting of $g_1g_2$. Choosing 
this lifting gives us
\[
\ol{m} = \ol{m}_1^{w_2} + \ol{m}_2 + \rho(w_1, w_2) 
   \quad \mbox{and hence} \quad
\ul{m} = \ul{m}_1^{w_2} + \ul{m}_2 + \eta_x(w_1,w_2) \, .
\]
So $mul(\ul{m}) = mul(\ul{m}_1)^{w_2} + mul(\ul{m}_2) + \eta_{x+1}(w_1,w_2)$, 
and therefore $\widehat{g_1g_2} = \widehat{g_1}\widehat{g_2}$ for this 
choice of liftings.

Now suppose that $\mu_g = \mu_{g'}$ in $Hom(A,B)$. Then $g' = cg$ for 
$c = g'g^{-1}$ in $C_{G_x}(A)$. We want $\mu_{\hat{g}} = \mu_{\hat{g}'}$, 
for which it suffices to show that $\hat{c} \in C_{G_{x+1}}(\hat{A})$. 
Let $A = C_{L,x}(\gamma) \times O$ and $c = (w,m)$. Let $h = (w,\ol{m})$ 
be a lift of $c$, then $w \in C_P(L)$; $O \leq C_{G_x}(w)$; and $h \in 
C_S(\ol{A})$ by Lemma~\ref{identity}\@. So $\zeta_{L,L,hg,x} = \gamma^w 
- \gamma$ by Lemma~\ref{zetax}(d), whence $\delta = 0$ in Theorem~\ref{morphx};
and $\zeta_{L,L,h} = \ol{\gamma}^w - \ol{\gamma}$ by Lemma~\ref{zeta}(d)\@.
So since
\begin{eqnarray*}
\zeta_{L,L,\hat{c},x+1} 
  & = & \hat{\zeta}_{L,L,w,x+1} 
                    - \lambda^w_{pro_{x+1}(\ol{m})+mul(\ul{m})}\\ 
  & = & pro_{x+1}(\hat{\zeta}_{L,L,w} 
                    - \lambda^w_{\ol{m}}) 
                    +\phi_{L,L,w,x+1}-\lambda^w_{mul(\ul{m})} \\ 
  & = & pro_{x+1}(\ol{\gamma}^w - \ol{\gamma}) + mul(\phi_{L,L,w,x}
                    -\lambda^w_{\ul{m}}) \quad \mbox{by (1) and 
                    construction of $\phi$} \\ 
  & = & pro_{x+1}(\ol{\gamma}^w - \ol{\gamma}) + mul(\ul{\gamma}^w
                    - \ul{\gamma}) \quad \mbox{by (2), since $\delta = 0$} \\ 
  & = & iso(\gamma)^w - iso(\gamma) \, ,
\end{eqnarray*}
$\hat{c}$ does indeed centralise. Finally, $1 \in S$ is a lifting of 
$1 \in G_x$ and thus $1 \in G_{x+1}$ is a pushout of $1 \in G_x$; that is, 
we can choose $\hat{1} = 1$.

Hence $F$ is well-defined on morphisms and satisfies $F(\mu_{g_1} \mu_{g_2})
= F(\mu_{g_1}) F(\mu_{g_2})$ as well as $F(id_A) = id_{F(A)}$. Thus, in 
summary, $F$ is a well-defined functor. 
\end{proof}

The following theorem also yields a proof for the main theorem of this
paper.

\begin{theorem}
Let $x \geq x_0$. Then $F$ is essentially surjective (or dense), full 
and faithful.
\end{theorem}

\begin{proof}
The functor $F$ is essentially surjective, as $F$ is a bijection on the 
objects of the underlying categories. Further, the functor $F$ is full 
and faithful, as it induces a bijection between the morphisms $Hom(A,B)$ 
with $Hom(F(A), F(B))$. More precisely, if $\mu_g$ is a morphism in 
$Hom(A,B)$, then $\mu_{\hat{g}}$ is a morphism in $Hom(F(A),F(B))$ and 
as we show below, this construction can be reversed. Hence we obtain the 
desired bijection.

To see that the construction can be reversed we take $g = (w,m) \in G_{x+1}$,
decompose $m$ as $m = pro_{x+1}(\ol{m}) + \ul{m}$, and want $\ul{m}$ to 
lie in the image of $mul$\@. In the proof of Theorem~\ref{shift} we have 
$x+1+e-2m \geq m+1$, and so the right hand side of (*) in that proof lies 
in $Z^1(L,p^{m+1}T/p^{x+1+e}T) = mul J^{1,*}(L,M_x)$. As the left hand side 
lies in $I^1(L,M_x)$, we conclude that $\ol{s}$ may be chosen to lie in 
$pT$, and hence $\ul{m}$ in the image of $mul$.
\end{proof}

\section{An example}
\label{examp}

We consider the Quillen categories for the 2-groups of coclass $1$. These
fall into 3 infinite families: the family of dihedral groups, the family
of quaternion groups and the family of semidihedral groups. All three
families are associated with the infinite pro-2-group $S$ defined by the
pro-2 presentation
\[ S = \langle a, b, t \mid a^2 = 1, b^2 = t, 
                            b^a = b t^{-1}, t^a = t^{-1}, t^b = t \rangle.\]

This infinite pro-2-group is a pro-2-analogue of the infinite dihedral
group. Note that the generator $t$ is redundant in the presentation; 
however, we use $T = \langle t \rangle = \gamma_2(S)$ as translation 
subgroup for our calculations and hence it is useful to include $t$ in
the presentation. 

%%%%%%%%%%%%%%%%%%%%%%%%%%%%%%%%%%%%%%%%%%%%%%%%%%%%%%%%%%%%%%%%%%%%%%%%%%%%%
\subsection{The infinite pro-2-group}

The group $P = S/T$ is abelian of order 4 and exponent 2. Let $\ol{a}, 
\ol{b}$ denote the images of $a$ and $b$ in $P$. Then $P = \langle \ol{a}, 
\ol{b} \rangle$ and as a first step we observe that $\L = \{ 
\langle 1 \rangle, \langle \ol{a} \rangle, \langle \ol{a} 
\ol{b} \rangle \}$ is the set of subgroups of $P$ which split over $T$. 
Based on this, we determine the set of objects in $\ol{\A}_p(S)$ as 
\[ \cT = \{ \langle 1 \rangle,
            \langle a \rangle, 
            \langle at \rangle, 
            \langle ab \rangle, 
            \langle abt \rangle \}. \]

We now list the morphisms in $\cT$. For each pair of objects $A$ and 
$B$ we list the morphisms $A \ra B$ by exhibiting a conjugating element.
We omit pairs $A$, $B$ with $Hom(A,B) = \emptyset$ and we omit the trivial
case $\langle 1 \rangle \ra A$ for each object $A$.

\begin{center}
\begin{tabular}{l|c}
source $\ra$ range & induced by \\
\hline
$\langle a \rangle \ra \langle a \rangle$ & $1$ \\
\hline
$\langle a \rangle \ra \langle at \rangle$ & $b$ \\
\hline
$\langle at \rangle \ra \langle at \rangle$ & $1$ \\
\hline
$\langle at \rangle \ra \langle a \rangle$ & $b t^{-1}$ \\
\hline
$\langle ab \rangle \ra \langle ab \rangle$ & $1$ \\
\hline
$\langle ab \rangle \ra \langle abt \rangle$ & $b$ \\
\hline
$\langle abt \rangle \ra \langle abt \rangle$ & $1$ \\
\hline
$\langle abt \rangle \ra \langle ab \rangle$ & $b t^{-1}$ \\
\hline
\end{tabular}
\end{center}

%%%%%%%%%%%%%%%%%%%%%%%%%%%%%%%%%%%%%%%%%%%%%%%%%%%%%%%%%%%%%%%%%%%%%%%%%%%%%
\subsection{The family of dihedral groups}

This family is defined by $\eta = 0$ and is described by the presentation
\[ G_x = \langle a, b, t \mid a^2 = 1, b^2 = t, t^{2^{x+2}} = 1,
                            b^a = b t^{-1}, t^a = t^{-1}, t^b = t \rangle.\]
Thus $G_x$ has order $2^{x+4}$. Let $M_x = \langle t \rangle$ of order 
$2^{x+2}$ and $O = \langle t^{2^{x+1}} \rangle = Z(G_x)$ of order $2$.
We note that $\L_\eta = \L$ and the set of objects in $\ol{\A}_p(G_x)$ is 
given by
\[ \cT_x = \{ \langle 1 \rangle,
              \langle a \rangle, 
              \langle at \rangle, 
              \langle ab \rangle, 
              \langle abt \rangle,
              O, 
              \langle a \rangle \times O, 
              \langle at \rangle \times O, 
              \langle ab \rangle \times O, 
              \langle abt \rangle \times O \}. \]

We list the morphisms between the objects in $\cT_x$ in the same format
as for $\cT$. We omit the trivial cases $\langle 1 \rangle \ra A$ and 
$O \ra A \times O$. We collect those cases on $Hom(A,B)$ together which 
have the same $Hom(\ol{A}, \ol{B})$ to facilitate an easy comparison of
liftings. 

\begin{center}
\begin{tabular}{l|c||l|c}
source $\ra$ range & induced by & 
source $\ra$ range & induced by \\
\hline
$\langle a \rangle 
    \ra \langle a \rangle$ & $1$ &
$\langle a \rangle 
    \ra \langle a \rangle \times O$ & $1$, $t^{2^x}$ \\
&& $\langle a \rangle \times O
    \ra \langle a \rangle \times O$ & $1$, $t^{2^x}$ \\
\hline

$\langle a \rangle 
    \ra \langle at \rangle$ & $b$ &
$\langle a \rangle 
    \ra \langle at \rangle \times O$ & $b$, $bt^{2^x}$ \\
&& $\langle a \rangle \times O
    \ra \langle at \rangle \times O$ & $b$, $b t^{2^x}$ \\
\hline

$\langle at \rangle 
    \ra \langle at \rangle$ & $1$ &
$\langle at \rangle 
    \ra \langle at \rangle \times O$ & $1$, $t^{2^x}$ \\
&& $\langle at \rangle \times O
    \ra \langle at \rangle \times O$ & $1$, $t^{2^x}$ \\
\hline

$\langle at \rangle 
    \ra \langle a \rangle$ & $b t^{-1}$ &
$\langle at \rangle 
    \ra \langle a \rangle \times O$ & $b t^{-1}$, $bt^{2^x-1}$ \\
&& $\langle at \rangle \times O
    \ra \langle a \rangle \times O$ & $b t^{-1}$, $bt^{2^x-1}$ \\
\hline

$\langle ab \rangle 
    \ra \langle ab \rangle$ & $1$ &
$\langle ab \rangle 
    \ra \langle ab \rangle \times O$ & $1$, $t^{2^x}$ \\
&& $\langle ab \rangle \times O
    \ra \langle ab \rangle \times O$ & $1$, $t^{2^x}$ \\
\hline

$\langle ab \rangle 
    \ra \langle abt \rangle$ & $b$ &
$\langle ab \rangle 
    \ra \langle abt \rangle \times O$ & $b$, $bt^{2^x}$ \\
&& $\langle ab \rangle \times O
    \ra \langle abt \rangle \times O$ & $b$, $bt^{2^x}$ \\
\hline

$\langle abt \rangle 
    \ra \langle abt \rangle$ & $1$ &
$\langle abt \rangle 
    \ra \langle abt \rangle \times O$ & $1$, $t^{2^x}$ \\
&& $\langle abt \rangle \times O
    \ra \langle abt \rangle \times O$ & $1$, $t^{2^x}$ \\
\hline

$\langle abt \rangle 
    \ra \langle ab \rangle$ & $b t^{-1}$ &
$\langle abt \rangle 
    \ra \langle ab \rangle \times O$ & $b t^{-1}$, $bt^{2^x-1}$ \\
&& $\langle abt \rangle \times O
    \ra \langle ab \rangle \times O$ & $b t^{-1}$, $b t^{2^x-1}$ \\
\hline
\end{tabular}
\end{center}

In this example the functor $F_S$ is essentially surjective. Further, for 
every two objects
$A$ and $B$ in $\ol{\A}_p(G_x)$, the map $Hom(A,B) \ra Hom(\ol{A}, \ol{B})$ 
induced by $F_S$ is surjective, but not necessarily injective. Hence
$F_S$ is full, but not faithful. 

\subsection{The family of semidihedral groups}

This family has $\eta \neq 0$ and is described by the presentation
\[ G_x = \langle a, b, t \mid a^2 = 1, b^2 = t, t^{2^{x+2}} = 1,
                            b^a = b t^{2^{x+e-1}-1}, t^a = t^{-1}, t^b = t 
         \rangle.\]

Thus $G_x$ has order $2^{x+4}$. Let $M_x = \langle t \rangle$ of order 
$2^{x+2}$ and $O = \langle t^{2^{x+1}} \rangle = Z(G_x)$ of order $2$.
In this example we find that $\L_\eta = \{ \langle 1 \rangle, \langle 
\ol{a} \rangle \}$. The set of objects in $\ol{\A}_p(G_x)$ is given by
\[ \cT_x = \{ \langle 1 \rangle,
            \langle a \rangle, 
            \langle at \rangle, 
            O, 
            \langle a \rangle \times O, 
            \langle at \rangle \times O \}. \]

We list the morphisms between the objects in $\cT_x$ in the same format
as for the dihedral groups. Again, we omit the trivial cases $\langle 1 
\rangle \ra A$ and $O \ra A \times O$. 

\begin{center}
\begin{tabular}{l|c||l|c}
source $\ra$ range & induced by & 
source $\ra$ range & induced by \\
\hline
$\langle a \rangle 
    \ra \langle a \rangle$ & $1$ &
$\langle a \rangle 
    \ra \langle a \rangle \times O$ & $1$, $t^{2^x}$ \\
&& $\langle a \rangle \times O
    \ra \langle a \rangle \times O$ & $1$, $t^{2^x}$ \\
\hline
$\langle a \rangle 
    \ra \langle at \rangle$ & $b t^{2^x}$ &
$\langle a \rangle 
    \ra \langle at \rangle \times O$ & $b$, $bt^{2^x}$ \\
&& $\langle a \rangle \times O
    \ra \langle at \rangle \times O$ & $b$, $bt^{2^x}$ \\
\hline
$\langle at \rangle 
    \ra \langle at \rangle$ & $1$ &
$\langle at \rangle 
    \ra \langle at \rangle \times O$ & $1$, $t^{2^x}$ \\
&& $\langle at \rangle \times O
    \ra \langle at \rangle \times O$ & $1$, $t^{2^x}$ \\
\hline
$\langle at \rangle 
    \ra \langle a \rangle$ & $b t^{-1}$ &
$\langle at \rangle 
    \ra \langle a \rangle \times O$ & $b t^{-1}$, $bt^{2^x-1}$ \\
&& $\langle at \rangle \times O
    \ra \langle a \rangle \times O$ & $b t^{-1}$, $bt^{2^x-1}$ \\
\hline
\end{tabular}
\end{center}

Hence in this case the functor $F_S$ is not essentially surjective, as 
$\L_\eta \neq \L$.
As in the example of the dihedral group, the functor $F_S$ is full, but
not faithful.

\subsection{The family of quaternion groups}

This family has $\eta \neq 0$ and is described by the presentation
\[ G_x = \langle a, b, t \mid a^2 = t^{2^{x+e-1}}, b^2 = t, t^{2^{x+2}} = 1,
                            b^a = b t^{-1}, t^a = t^{-1}, t^b = t 
         \rangle.\]

Thus $G_x$ has order $2^{x+4}$. Let $M_x = \langle t \rangle$ of order 
$2^{x+2}$ and $O = \langle t^{2^{x+1}} \rangle = Z(G_x)$ of order $2$.
In this example we find that $\L_\eta = \{ \langle 1 \rangle \}$ and the
set of objects in $\ol{\A}_p(G_x)$ is given by
\[ \cT_x = \{ \langle 1 \rangle, O \}. \]

Hence $\ol{\A}_p(G_x)$ contains only the trivial morphism $\langle 1 \rangle
\ra O$ in this case.

\end{document}